\documentclass[10pt,leqno]{article} 

   \usepackage[centertags]{amsmath}
   \usepackage{amsfonts}
   \usepackage{amsmath}
   \usepackage{amssymb}
   \usepackage{amsthm}
   \usepackage{newlfont}
   \usepackage[latin1]{inputenc}

 \usepackage{multirow, bigdelim}

\usepackage{color}

\allowdisplaybreaks[3]

\theoremstyle{plain}
\newtheorem{theorem}{Theorem}[section]
\newtheorem{lemma}[theorem]{Lemma}

\theoremstyle{definition}
\newtheorem{definition}[theorem]{Definition}

\theoremstyle{remark}
\newtheorem{remark}[theorem]{Remark}
\newtheorem*{remark*}{Remark}

\numberwithin{equation}{section}

\newcommand{\dosfilas}[2]{
  \ldelim[{2}{2mm}& #1 &\rdelim]{2}{2mm} \\
  & #2 & &  & &
}

\newcommand*\pFqskip{8mu}
\catcode`,\active
\newcommand*\pFq{\begingroup
        \catcode`\,\active
        \def ,{\mskip\pFqskip\relax}%
        \dopFq
}
\catcode`\,12
\def\dopFq#1#2#3#4#5{%
        {}_{#1}F_{#2}\biggl(\genfrac..{0pt}{}{#3}{#4};#5\biggr)%
        \endgroup
}

\newcommand\D{{\mathcal D}}

\newcommand\Hh{{\mathsf h}}
\newcommand\Pp{{\mathsf P}}
\newcommand\M{{\mathcal M}}
\newcommand\ax{{\mathsf{a}}}
\newcommand\bx{{\mathsf{b}}}
\newcommand\Nn{{\hat N}}

\newcommand\CC{{\mathbb C}}

\newcommand\RR{{\mathbb R}}
\newcommand\ZZ{{\mathbb Z}}
\newcommand\NN{{\mathbb N}}
\newcommand\PP{{\mathbb P}}
\newcommand\Aa{{\mathbb A}}
\newcommand\Bb{{\mathbb B}}

\newcommand\supp{\operatorname{Supp}}

\newcommand\sign{\operatorname{sign}}

\newcommand\Sh{\mbox{\Large $\mathfrak {s}$}}

   \title{Exceptional Hahn and Jacobi polynomials with an arbitrary number of continuous parameters.
   \footnote{Partially supported by PGC2018-096504-B-C31
(FEDER(EU)/Ministerio de Ciencia e Innovaci\'on-Agencia Estatal de Investigaci\'on),
FQM-262 and Feder-US-1254600 (FEDER(EU)/Jun\-ta de Anda\-lu\-c\'ia).}}
   \author{Antonio J. Dur\'{a}n\\
     \footnotesize
        \  Departamento de An\'{a}lisis Matem\'{a}tico.
       Universidad de Sevilla \\
       \footnotesize Apdo (P. O. BOX) 1160. 41080 Sevilla. Spain.
   duran@us.es \\
          \ \ }
   \date{}
   
\begin{document}
   \maketitle

\bigskip

\begin{abstract}
We construct new examples of exceptional Hahn and Jacobi polynomials.
Exceptional polynomials are orthogonal polynomials with respect to a measure which are also eigenfunctions of a second order difference or differential operator.  The most apparent difference between classical or classical discrete orthogonal polynomials and their exceptional counterparts
is that the exceptional families have gaps in their degrees, in the sense that not all degrees are present in the sequence of polynomials. The new examples have the novelty that they depend on an arbitrary number of continuous parameters.
\end{abstract}

\section{Introduction and results}
Exceptional and exceptional discrete orthogonal polynomials $p_n$, $n\in X\varsubsetneq \NN$, with $\NN\setminus X$ a finite set, are complete orthogonal polynomial systems with respect to a positive measure which in addition are eigenfunctions of a second order differential or difference operator, respectively. They extend the  classical families of Hermite, Laguerre and Jacobi or the classical discrete families of Charlier, Meixner and Hahn.

The last decade and a half has seen a great deal of activity in the area of exceptional orthogonal polynomials (see, for instance,
\cite{Be,BK,duch,dume,duha,GFGM,GUKM1,GUKM2} (where the adjective \textrm{exceptional} for this topic was introduced),  \cite{GUGM,GQ,MR,OS3,OS4,STZ}, and the references therein). The most apparent difference between classical or classical discrete orthogonal polynomials and their exceptional counterparts
is that the exceptional families have gaps in their degrees, in the
sense that not all degrees are present in the sequence of polynomials (as it happens with the classical families) although they form a complete orthonormal set of the underlying $L^2$ space defined by the orthogonalizing positive measure.

In all the examples appeared before 2021 apart from the parameters associated to
the classical and classical discrete weights, only discrete parameters appear in the construction of each exceptional family.
This scenario has recently changed. Indeed, in \cite{xle}, M.A. Garc\'\i a Ferrero, D. G\'omez-Ullate and R. Milson have introduced exceptional Legendre polynomials depending on an arbitrary number of continuous parameters.

The purpose of this paper is to construct new examples of exceptional Hahn and Jacobi polynomials depending on an arbitrary number of continuous parameters. We use the same approach than in our previous paper \cite{duch,dume,duha}, and hence we construct new families of exceptional  Hahn polynomials
by dualizing the examples of Krall dual Hahn polynomials introduced in \cite{dundh} and which depend on an arbitrary number of continuous parameters. Krall or Krall discrete polynomials $q_n$, $n\ge 0$, are orthogonal polynomials which are eigenfunctions of a higher order differential or difference operator, respectively. Krall polynomials we introduced  more than eighty years ago when H.L. Krall raised the issue of orthogonal polynomials which are also common eigenfunctions of a higher order differential operator. He obtained a complete classification for the case of a differential operator of order four (\cite{Kr2}).
Since the eighties a lot of effort has been devoted to find Krall polynomials (\cite{du0,du1,dudh,DdI,DdI2,DdI3,DdI4,koekoe,koe,koekoe2,L1,L2,GrH1,GrHH,GrY,Plamen1,Plamen2,Zh}, the list is by no mean exhaustive).

Our starting point is the following example of exceptional Hahn and Jacobi polynomials. For $\alpha,\beta$ real numbers with $\alpha,\beta\not =-1,-2,\cdots$, and $N$ a positive integer let $h_n^{\alpha, \beta, N}$, $P_n^{\alpha, \beta}$ be the $n$-th Hahn and Jacobi polynomial, respectively (see (\ref{hpol}) and (\ref{defjac}) below). For a finite set $F$ of positive integers, consider the following polynomials (of degree $n$)
\begin{align}\label{xh1}
h_n^{\alpha, \beta, N;F}(x)&=\left|
  \begin{array}{@{}c@{}lccc@{}c@{}}
    & h_{n-u_F}^{\alpha ,\beta ,N}(x+j-1) &&\hspace{-.4cm}{}_{1\le j\le n_F+1} \\
    \dosfilas{ h_{f}^{\alpha ,\beta ,N}(x+j-1) }{f\in F} \\
      \end{array}
  \hspace{-.3cm}\right|,\\\label{xj1}
P_n^{\alpha,\beta ;F}(x)&= \left|
  \begin{array}{@{}c@{}lccc@{}c@{}}
    & (P_{n-u_F}^{\alpha ,\beta })^{(j-1)}(x) &&\hspace{-.4cm}{}_{1\le j\le n_F+1} \\
    \dosfilas{(P_{f}^{\alpha ,\beta })^{(j-1)}(x) }{f\in F}
  \end{array}
  \hspace{-.3cm}\right|,
\end{align}
where $n\in \sigma_F$ and
\begin{equation}\label{elsig}
\sigma_F=\{u_F,u_F+1,\cdots\}\setminus \{u_F+f:f\in F\},\quad u_F=\sum_{f\in F}f-\binom{n_F+1}{2}
\end{equation}
(the examples (\ref{xh1}) and (\ref{xj1}) are the case $F_2=\emptyset$ in \cite{duha}).

Along this paper, we use the following notation:
given a finite set of positive integers $F=\{f_1,\ldots , f_{n_F}\}$, the expression
\begin{equation}\label{defdosf}
  \begin{array}{@{}c@{}lccc@{}c@{}}
  &  &&\hspace{-.9cm}{}_{1\le j\le n_F} \\
    \dosfilas{ z_{f,j}  }{f\in F}
  \end{array}
\end{equation}
inside of a matrix or a determinant will mean the submatrix defined by
$$
\left(
\begin{array}{cccc}
z_{f_1,1} & z_{f_1,2} &\cdots  & z_{f_1,n_F}\\
\vdots &\vdots &\ddots &\vdots \\
z_{f_{n_F},1} & z_{f_{n_F},2} &\cdots  & z_{f_{n_F},{n_F}}
\end{array}
\right) .
$$
The determinants (\ref{xh1}) and (\ref{xj1}) should be understood in this form.
If $X$ is a finite set, we denote by $n_X$ the number of elements of $X$.

It was proved in \cite{duha}, that the polynomials $h_n^{\alpha, \beta, N;F}$, $n\in \sigma_F$, are eigenfunctions of a second order difference operator, while the polynomials $P_n^{\alpha, \beta;F}$, $n\in \sigma_F$, are eigenfunctions of a second order differential operator. Under certain admissibility conditions on $\alpha$, $\beta$ and $F$ both sequences of polynomials are orthogonal with respect to positive measures. For instance, that is the case when $\alpha,\beta >-1$ and $\prod_{f\in F}(x-f)\ge 0$, $x\in \NN$. In this paper, the families (\ref{xh1}) and (\ref{xj1}) are called standard examples.

The cases $\alpha,\beta = -1,-2,\cdots$ were not considered in \cite{duha} (and, as far as this author knows, in any other paper on exceptional polynomials)
because some of the Hahn and Jacobi polynomials collapse to zero and then both determinants (\ref{xh1}) and (\ref{xj1}) collapse also to zero. Apparently this degeneracy has the consequence that the cases $\alpha,\beta = -1,-2,\cdots$ seem to have little interest. However one should take into account that appearances can be very deceiving! Indeed, in the new examples of exceptional Hahn and Jacobi polynomials constructed in this paper the parameters $\alpha$ and $\beta$ are taken to be negative integers. By choosing the finite set $F$ appropriately, we show that the degeneracy can be avoided and a pletora of new examples of exceptional Hahn and Jacobi polynomials, depending now of an arbitrary number of continuous parameters, can be constructed. More precisely, consider two negative integers $\ax,\bx$ and a positive integer $N$ satisfying $-N\le \ax\le \bx\le -1$ (we use the notation $\ax,\bx$ instead of the usual $\alpha,\beta$ to stress that the numbers $\ax$ and $\bx$ are negative integers). Let $F$ be a finite set of positive integers satisfying
\begin{equation}\label{cis}
\{-\bx,\cdots, -\ax-\bx-1\}\subset F.
\end{equation}
As explained above, the formulas (\ref{xh1}) and (\ref{xj1}) does not work because   $h_{f}^{\ax,\bx,N}=P_{f}^{\ax,\bx}=0$, $f\in \{-\ax,\cdots ,-\ax-\bx-1\}\subset F$, and then $h_n^{\ax,\bx,N;F}=P_n^{\ax,\bx;F}=0$, $n\ge 0$. But we can fix this problem by substituting some of the Hahn polynomials $h_n^{\alpha,\beta,N}$ in (\ref{xh1}) or some of the Jacobi polynomials $P_n^{\alpha,\beta}$ in (\ref{xj1}) by some relative families of polynomials $\Hh_n^{\ax,\bx,N;\M}$ and $\Pp_n^{\ax,\bx;\M}$, respectively. In fact, we have found such families which it turns out to depend on a finite set of $-\bx$ real parameters. Miraculously, everything then works as in the standard examples: the new families are eigenfunctions of a second order difference or differential operator, respectively, and, formally these operators are identical to the operators of the standard families. And there is a surprisingly simple admissibility condition for the new families of exceptional Hahn and Jacobi polynomials to be orthogonal with respect to positive measures. These measures are of the same form as the orthogonalizing measures of the standard families. The proofs of these results are however much more complicated, and some of then have needed a different approach to the one used for the standard families (for instance, we can not use the Christoffel transform machinery as in \cite{duch,dume,duha} because the new Krall dual Hahn families constructed in \cite{dundh} are not anymore Christoffel transform of the dual Hahn measure).

The content of this paper is as follows.

In Section \ref{sec3}, we construct new families of exceptional Hahn polynomials depending on an arbitrary number of parameters.
We denote by
$$
\M=\{M_0,M_1,M_2, \cdots \}
$$
a set consisting of real parameters $M_i$ with $M_i\not =0,1$,
and consider two negative integers $\ax,\bx$  satisfying $\ax\le \bx\le -1$ and a real number $N\not=0,-1,\cdots$.

We need to introduce some auxiliary functions. As usual, $\lceil x\rceil$ denotes the ceiling function: $\lceil x\rceil=\min \{n\in \ZZ:n\ge x\}$, and $(x)_m$, $m\in \NN$, denotes the Pochhammer symbol $(x)_m=x(x+1)\cdots (x+m-1)$; we also set $(x,y)_m=(x)_m(y)_m$.
For $u\in \NN$, $u\le -\ax-1$, we define
$$
\varphi_{u}^{\ax,\bx,N}(s,x)=(\ax+1,-N)_{\max(u,-\ax-\bx-u-1)}\pFq{3}{2}{-u,u-s+\ax+\bx+1,-x}{\ax-s+1,-N}{1}.
$$
Since $u\in \NN$, except for  normalization, $\varphi_{u}^{\ax,\bx,N}(s,x)$ is the Hahn polynomial $h_u^{\ax-s,\bx,N}(x)$. Hence as a function of $x$ $\varphi_{u}^{\ax,\bx,N}(s,x)$ is a polynomial of degree at most $u$, and as a function of $s$ it is rational and analytic at $s=0$ when $u\le -\ax-1$. We next define the sequence of polynomials $(\Hh _n^{\ax,\bx,N;\M})_n$ which are going to play the role of the Hahn polynomials in the new examples of exceptional Hahn polynomials.

\begin{definition}\label{losw} Let $\ax,\bx$ and $N$ be two negative integers satisfying $\ax\le \bx\le -1$ and a real number $N\not=0,-1,\cdots$.
We define the sequence $(\Hh _n^{\ax,\bx,N;\M})_n$ of polynomials, $\Hh_n^{\ax,\bx,N;\M}$ of degree $n$, as follows.

\noindent
For $\lceil\frac{-\ax-\bx}{2}\rceil \le n\le -\ax-1$
\begin{equation}\label{losw1}
\Hh _n^{\ax,\bx,N;\M}(x)=\frac{\partial}{\partial s}\varphi_{n}^{\ax,\bx,N}(0,x)-\frac{\partial}{\partial s}\varphi_{-\ax-\bx-n-1}^{\ax,\bx,N}(0,x);
\end{equation}
for $-\ax\le n\le -\ax-\bx-1$
\begin{align}\label{losw2}
\Hh _n^{\ax,\bx,N;\M}(x)&=(-1)^{\bx+n}(n+\bx)!\Big[(-\ax-\bx-n-1)!(-x)_{-\ax}h_{\ax+n}^{-\ax,\bx,\ax+N}(x+\ax)\\\nonumber &\quad \quad+\frac{(n+\ax)!(-N-\ax-\bx-n-1)_{2n+\ax+\bx+1}}{M_{\ax+n}-1}h_{-\ax-\bx-n-1}^{\ax,\bx,N}(x)\Big] ;
\end{align}
otherwise
\begin{equation}\label{losw3}
\Hh _n^{\ax,\bx,N;\M}(x)=h_n^{\ax,\bx,N}(x)
\end{equation}
(as before $h_n^{\ax,\bx,N}$ denotes the $n$-th Hahn polynomial, see (\ref{hpol}) below).
\end{definition}

Notice that only the polynomials $\Hh_{n}^{\ax,\bx,N;\M}$, $-\ax\le n\le -\ax-\bx-1$, depend on the parameters in $\M$, more precisely: only the polynomial
$\Hh_{i-\ax}^{\ax,\bx,N;\M}$ depends on the parameter $M_i$, $i=0,\cdots , -\bx-1$. We introduced these polynomials in \cite{dundh}, but we will explain in Section \ref{sec1} how these auxiliary polynomials $(\Hh_n^{a,b,N;\M})_n$ can be constructed by taking limit in a suitable way in (\ref{xh1}).

The new families of exceptional Hahn polynomials $h_n^{\ax,\bx,N;\M, F}$, $n\in \sigma _F$ (\ref{elsig}), are defined  by
\begin{equation}\label{defmexi}
h_n^{\ax, \bx, N;\M, F}(x)=\left|
  \begin{array}{@{}c@{}lccc@{}c@{}}
    & \Hh_{n-u_F}^{\ax ,\bx ,N;\M}(x+j-1) &&\hspace{-.4cm}{}_{1\le j\le n_F+1} \\
    \dosfilas{\Hh_{f}^{\ax ,\bx ,N;\M}(x+j-1) }{f\in F} \\
  \end{array}
  \hspace{-.4cm}\right|,
\end{equation}
where  $(\Hh_n^{\ax,\bx;\M})_n$ are the polynomials introduced in Definition \ref{losw}.
Using Lemma 3.4 of \cite{DdI}, we deduce that the polynomial $h_n^{\ax, \bx, N;\M, F}$ has degree $n$
for $n\in \sigma_F$. The sequence of polynomials $h_n^{\ax,\bx,N;\M, F}$, $n\in \sigma _F$, depend on the $-\bx-n_-$ parameters $b_i$, $i\in F_{\bx}$, where
\begin{equation}\label{conp}
F_\bx=\{0,1,\cdots, -\bx-1\}\setminus \{-\bx-f-1:f\in F\},
\end{equation}
and $n_-$ is the number of positive integers in $F$ which are less than $-\bx$.

As mentioned above, we study the polynomials $(h_n^{\ax, \bx, N;\M, F})_n$
by dualizing the orthogonal polynomials with respect to the Krall dual Hahn measures constructed in \cite{dundh}.
In order to introduce here these measures we assume $N$ to be a positive integer with $-N\le \ax\le \bx\le -1$ and
that the finite set $F$ of positive integers satisfies (\ref{cis}). We adapt the notation to that of \cite{dundh} and set
\begin{equation}\label{mn}
a=-\ax,\quad b=-\bx,\quad \Nn=N+\ax+\bx,
\end{equation}
so that $1\le a,b\le N$. Consider finally the finite set of integers $U_F$  defined by
\begin{align}\label{elu}
U_F&=U_{F_-}\cup U_{F_+},\\\label{elu1}
U_{F_-}&=\{f+\ax+\bx:\mbox{$f\in F$ and $1\le f\le -\bx-1$}\},\\\label{elu2}
U_{F_+}&=\{f+\ax+\bx:\mbox{$f\in F$ and $-\ax-\bx\le f$}\}.
\end{align}
We then define the measures $\nu_{a,b,\Nn}^{\M,U_F}$ and $\nu_{a,b,\Nn}^{\M}$ by
\begin{align}\label{lctnu}
\nu_{a,b,\Nn}^{\M,U_F}&=\prod_{u\in U_F}(x-\lambda^{a,b}(u))\nu_{a,b,\Nn}^{\M},\\\label{lanu}
\nu_{a,b,\Nn}^{\M}&=\sum_{x=-b}^{-1}\frac{(2x+a+b+1)(\Nn+1-x)_{x+b}}{(\Nn+b+1)_{x+a+1}}M_{x+b}\delta_{\lambda^{a,b}(x)}\\\nonumber
&\qquad +\frac{(\Nn+1)_b^2}{(b+1)_{a-b}}\sum_{x=0}^\Nn \frac{\rho_{b,a,\Nn}(x)}{\prod_{i=0}^{b-1}(x+a+i+1)(x+b-i)}\delta _{\lambda^{a,b}(x)},
\end{align}
where $\rho_{b,a,N}$ is the dual Hahn measure (see (\ref{masdh}) below) and
\begin{equation}\label{deflamb}
\lambda^{a,b}(x)=x(x+a+b+1).
\end{equation}
Note that the measure $\nu_{a,b,\Nn}^{\M}$ depends on the parameters $M_i$, $i=0,\cdots, -\bx-1$, and it is positive if and only if these parameters are positive.
However, the measure $\nu_{a,b,\Nn}^{\M,U_F}$ depends on the parameters $M_i$, $i\in F_\bx$ (see (\ref{conp})) because each integer $u\in U_{F_-}$ kills the mass at $\lambda^{a,b}(-u-\ax-\bx-1)$ of the measure $\nu_{a,b,\Nn}^{\M,U_F}$ (note that $\lambda^{a,b}(-u-\ax-\bx-1)=\lambda^{a,b}(u)$ and $-\bx\le -u-\ax-\bx-1\le -2$ when $u\in U_{F_-}$).

In Lemma \ref{lem3.2}, we prove that the sequence of orthogonal polynomials with respect to the measure (\ref{lctnu}) and $h_n^{\ax,\bx,N;\M, F}$ are dual sequences.

As a consequence, we show in Theorem \ref{th3.3} that the polynomials $h_n^{\ax,\bx,N;\M,F}$, $n\in \sigma _F$, are eigenfunctions of a second order difference operator $D$, whose coefficients are rational functions (and which correspond to the coefficients of the three term recurrence formula for the orthogonal polynomials with respect to the measure $\nu_{a,b,\Nn}^{\M,U_F}$).

The most interesting case appears when the measure $\nu_{a,b,\Nn}^{\M,U_F}$ is positive. This gives rise to the concept of admissibility:

\begin{definition}\label{laad} We say that $\ax,\bx$ ($\ax\le\bx\le -1$), $\M$ and $F$ (satisfying (\ref{cis})) are admissible if the two following conditions holds
\begin{enumerate}
\item $\displaystyle\sign M_i=\sign \big[ \prod_{f\in F_{\textrm{ext}}}(i-f-\ax)(i+f+\bx+1)\big]$,\quad $i\in F_\bx$ (\ref{conp}), where
$F_{\textrm{ext}}=F\setminus \{-\bx,\cdots,-\ax-\bx-1\}$.
\item $\displaystyle\prod_{f\in F;f\ge-\ax-\bx}(x-f-\ax-\bx)\ge 0$, \quad $x=0,\cdots, \max \{f+\ax+\bx:f\in F\}$.
\end{enumerate}
\end{definition}
It is not difficult to see that $\ax,\bx$, $\M$ and $F$ are admissible if and only if the measure $\nu_{a,b,\Nn}^{\M,U_F}$ is positive.

In Lemma \ref{l3.1} we prove that this admissibility condition is equivalent to
$$
\Omega _{\M,F}^{\ax, \bx, N}(n)\Omega _{\M,F}^{\ax, \bx, N}(n+1)>0,\quad n=0,\cdots,N-n_F,
$$
where $\Omega _{\M, F}^{\ax, \bx, N}$ is the polynomial defined by
\begin{equation}\label{defom}
\Omega _{\M,F}^{\ax, \bx, N}(x)=\left|
  \begin{array}{@{}c@{}lccc@{}c@{}}
    &  &&\hspace{-.9cm}{}_{1\le j\le n_F} \\
    \dosfilas{\Hh_{f}^{\ax ,\bx ,N}(x+j-1) }{f\in F}
  \end{array}
  \hspace{-.3cm}\right|.
\end{equation}
The admissibility condition in Definition \ref{laad} allows us to define the positive measure
$$
\omega_{\ax,\bx,N}^{\M, F}=\sum_{x=0}^{N-n_F} \frac{\binom{\ax +n_F+x}{x}\binom{\bx +N-x}{N-n_F-x}}{\Omega_{\M, F}^{\ax,\bx,N}(x)\Omega_{\M, F}^{\ax,\bx,N}(x+1)}\delta_x.
$$
In Theorem \ref{th4.5} we prove that, under the assumption of the admissibility condition in Definition \ref{laad}, the polynomials $h_n^{\ax,\bx,N;\M, F}$, $n\in \sigma_F, n\le N+n_F$, are orthogonal and complete with respect to the positive measure $\omega_{\ax,\bx,N}^{\M, F}$.

We complete Section \ref{sec3} showing how to remove the assumption $\ax\le \bx$.

In Section \ref{sec5}, we construct new sequences of exceptional Jacobi polynomials depending of an arbitrary number of continuous parameters. We do that by taking limits in the exceptional Hahn families constructed in Section \ref{sec3}.

For $u\in \NN$, $u\le -\ax-1$, we define
$$
\varphi_{u}^{\ax,\bx}(s,x)=\frac{(\ax+1)_{\max(u,-\ax-\bx-u-1)}\pFq{2}{1}{-u,u-s+\ax+\bx+1}{\ax-s+1}{(1-x)/2}}{\max(u!,(-\ax-\bx-u-1)!)}.
$$
Except for the normalization constant in front of the hypergeometric function, and since $u\in \NN$, $\varphi_{u}^{\ax,\bx}(s,x)$ is the Jacobi polynomial $P_u^{\ax-s,\bx}(x)$. Hence, as a function of $x$ $\varphi_{u}^{\ax,\bx}(s,x)$ is a polynomial of degree at most $u$, and as a function of $s$ it is rational and analytic at $s=0$ when $u\le -\ax-1$. We next define the sequence of polynomials $(\Pp _n^{\ax,\bx;\M})_n$ which are going to play the role of the Jacobi polynomials in the new examples of exceptional Jacobi polynomials.

\begin{definition}\label{losj} Let $\ax,\bx$ be negative integers satisfying $\ax\le\bx\le -1$.
We define the sequence $(\Pp _n^{\ax,\bx;\M})_n$ of polynomials, $\Pp_n^{\ax,\bx;\M}$ of degree $n$, as follows.

\noindent
For $\lceil\frac{-\ax-\bx}{2}\rceil \le n\le -\ax-1$
\begin{equation}\label{losj1}
\Pp _n^{\ax,\bx;\M}(x)=\frac{\partial}{\partial s}\varphi_{n}^{\ax,\bx}(0,x)-\frac{\partial}{\partial s}\varphi_{-\ax-\bx-n-1}^{\ax,\bx}(0,x);
\end{equation}
for $-\ax\le n\le -\ax-\bx-1$
\begin{align}\label{losj2}
\Pp _n^{\ax,\bx,N;\M}(x)&=\frac{(n+\ax)!(n+\bx)!(-\ax-\bx-n-1)!}{(-1)^{\bx+n}n!}\\\nonumber &\quad\quad \times\Big[\frac{1}{M_{\ax+n}-1}P_{-\ax-\bx-n-1}^{\ax,\bx}(x) +\left(\frac{1-x}{2}\right)^{-\ax}P_{\ax+n}^{-\ax,\bx}(x)\Big];
\end{align}
otherwise
\begin{equation}\label{losj3}
\Pp _n^{\ax,\bx;\M}(x)=P_n^{\ax,\bx}(x)
\end{equation}
(as before $P_n^{\ax,\bx}$ denotes the $n$-th Jacobi polynomial, see (\ref{defjac}) below).
\end{definition}

Notice that again, only the polynomial $\Pp_{i-\ax}^{\ax,\bx;\M}$ depends on the parameter $M_i$, $i=0,\cdots , -\bx-1$.

The polynomials $(\Pp _n^{\ax,\bx;\M})_n$ can be obtained in two different ways. We explain in Section \ref{sec1} how these auxiliary polynomials $(\Pp_n^{a,b;\M})_n$ can be constructed by taking limit in a suitable way in (\ref{xj1}). But also, the polynomial $\Pp_n^{a,b;\M}$ can be produced by changing $x\to (1-x)N/2$ in $\Hh_n^{\ax,\bx,N;\M}$ and taking limit when $N\to \infty$ (i.e., in the same way as the $n$-th Jacobi polynomial can be produced from the $n$-th Hahn polynomial).

The new families of exceptional Jacobi polynomials $P_n^{\ax,\bx;\M, F}$, $n\in \sigma _F$ (\ref{elsig}), are defined  by
\begin{equation}\label{losjx}
P_n^{\ax, \bx;\M, F}(x)=\left|
  \begin{array}{@{}c@{}lccc@{}c@{}}
    & (\Pp_{n-u_F}^{\ax ,\bx;\M})^{(j-1)}(x) &&\hspace{-.4cm}{}_{1\le j\le n_F+1} \\
    \dosfilas{(\Pp_{f}^{\ax ,\bx ;\M})^{(j-1)}(x) }{f\in F} \\
  \end{array}
  \hspace{-.3cm}\right|,
\end{equation}
where  $(\Pp_n^{\ax,\bx;\M})_n$ are the polynomials introduced in Definition \ref{losj}.
The polynomial $P_n^{\ax, \bx;\M, F}$ has degree $n$ for $n\in \sigma_F.$
As for the exceptional Hahn family, the sequence of polynomials $P_n^{\ax,\bx,N;\M, F}$, $n\in \sigma _F$, depend on the $-\bx-n_-$ parameters $M_i$, $i\in F_{\bx}$ (\ref{conp}).

Assuming that the finite set $F$ of positive integers satisfies (\ref{cis}), we prove in Theorem \ref{th5.1} that the polynomials $P_n^{\ax,\bx;\M,F}$, $n\in \sigma _F$, are eigenfunctions of a second order differential operator, whose coefficients are rational functions.

The most interesting case appears when $\ax,\bx$, $\M$ and $F$ are admissible (Definition \ref{laad}). Admissibility is closely related to the fact that
the polynomial
\begin{equation}\label{defomj}
\Omega _{\M, F}^{\ax, \bx}(x)=\left|
  \begin{array}{@{}c@{}lccc@{}c@{}}
    &  &&\hspace{-.9cm}{}_{1\le j\le n_F} \\
    \dosfilas{(\Pp_{f}^{\ax ,\bx})^{(j-1)}(x) }{f\in F}
  \end{array}
  \hspace{-.3cm}\right|
\end{equation}
has not roots in $[-1,1]$. In fact we prove that if
\begin{equation}\label{lamie}
\Omega _{\M, F}^{\ax, \bx}(x)\not =0,\quad x\in [-1,1],
\end{equation}
then $\ax,\bx$, $\M$ and $F$ are admissible. We have computation evidence showing that the converse is also true, but we have not been able to prove it. Hence, we propose it here as a conjecture.

\medskip

\noindent\textit{Conjecture}.  If $\ax,\bx$ ($\ax\le\bx\le -1$), $\M$ and $F$ (satisfying (\ref{cis})) are admissible (see Definition \ref{laad}) then
$\Omega _{\M, F}^{\ax, \bx}(x)\not =0$, $x\in [-1,1]$.

If (\ref{lamie}) holds, we prove that the polynomials $P_n^{\ax,\bx;\M,F}$, $n\in \sigma _F$, are orthogonal and complete with respect to the positive weight in $[-1,1]$ defined by
$$
\omega_{\alpha,\beta;\M, F} =\frac{(1-x)^{\ax +n_F}(1+x)^{\bx+n_F}}{(\Omega_{\M, F}^{\ax,\bx}(x))^2}
$$
(see Theorem \ref{th5.6}).

We complete  Section \ref{sec5} showing how to remove the assumption $\ax\le \bx$, and comparing our examples with that constructed in \cite{xle}. We point out that the exceptional Legendre polynomials are constructed in \cite{xle} using a completely different approach (the so-called confluent Darboux transformation is the main tool) and that neither the polynomials $(\Pp_{n}^{\ax ,\bx ;\M})_n$ (\ref{losj}) nor the Wronskian determinant (\ref{losjx}) are used in \cite{xle}.

In the last Section, we consider the case when the finite set $F$ does not satisfy (\ref{cis}).

\section{Preliminares}
In Section \ref{sec3} of this paper, we deal with discrete measures supported in a finite number of mass points. The following lemma will be useful to manage these measures.

\begin{lemma}[Lemma 2.1 of \cite{dudh}]\label{ldmp} Consider a discrete measure $\mu =\sum_{i=0}^N\mu_i\delta_{x_i}$, with $\mu_i\not =0$, $i=0,\cdots , N$.
\begin{enumerate}
\item If we assume that there exists a sequence $p_i$, $i=0,\cdots, N$, of orthogonal polynomials, with $\deg (p_i)=i$ and such that $\langle p_i,p_i\rangle\not =0$ has constant sign, then either $\mu_i>0$ or $\mu_i<0$, $i=0,\cdots, N$.
\item If we assume that there exists a sequence $(f_i)_{i=0}^{N+1}$ of orthogonal functions with non-null $L^2$ norm, then these functions form a basis of $L^2(\mu)$.
\end{enumerate}
\end{lemma}

We also will need the Sylvester's determinant identity (for the proof and a more general formulation of the Sylvester's identity see \cite{Gant}, p. 32).

\bigskip
\begin{lemma}\label{lemS}
For a square matrix $M=(m_{i,j})_{i,j=1}^k$,  and for each $1\le i, j\le k$, denote by $M_i^j$ the square matrix that results from $M$ by deleting the $i$-th row and the $j$-th column. Similarly, for $1\le i, j, p,q\le k$ denote by $M_{i,j}^{p,q}$ the square matrix that results from $M$ by deleting the $i$-th and $j$-th rows and the $p$-th and $q$-th columns.
The Sylvester's determinant identity establishes that for $i_0,i_1, j_0,j_1$ with $1\le i_0<i_1\le k$ and $1\le j_0<j_1\le k$, then
$$
\det(M) \det(M_{i_0,i_1}^{j_0,j_1}) = \det(M_{i_0}^{j_0})\det(M_{i_1}^{j_1}) - \det(M_{i_0}^{j_1}) \det(M_{i_1}^{j_0}).
$$
\end{lemma}

Given a finite set of numbers $X=\{x_1,\cdots, x_{n_X}\}$, $x_i<x_j$ if $i<j$, we denote by $V_X$ the Vandermonde determinant defined by
\begin{equation}\label{defvdm}
V_X=\prod_{1=i<j=k}(x_j-x_i).
\end{equation}

\subsection{Dual Hahn, Hahn and Jacobi polynomials}
We include here basic definitions and facts about dual Hahn, Hahn and Jacobi polynomials, which we will need in the following Sections.

We write $(R_{n}^{a,b,N})_n$ for the sequence of dual Hahn polynomials defined by
\begin{equation}\label{dhpol}
R_{n}^{a,b,N}(x)=\sum _{j=0}^n\frac{(-n)_j(-N+j)_{n-j}(a+j+1)_{n-j}}{n!(-1)^j j!}\prod_{i=0}^{j-1}(x-i(a+b+1+i))
\end{equation}
(see \cite{KLS}, pp, 209-13).
We have taken a different normalization that in \cite{dudh} since
we deal here with the case when $a$ is a negative integer.

Notice that $R_{n}^{a,b,N}$ is always a polynomial of degree $n$.
Using that
$$
(-1)^j\prod_{i=0}^{j-1}(\lambda^{a,b}(x)-i(a+b+1+i))=(-x)_j(x+a+b+1)_j,
$$
we get the hypergeometric representation
$$
R_{n}^{a,b,N}(\lambda^{a,b}(x))=\frac{(a+1)_n(-N)_n}{n!}\pFq{3}{2}{-n,-x,x+a+b+1}{a+1,-N}{1},
$$
where $\lambda^{a,b}(x)$ is defined by (\ref{deflamb})

The dual Hahn polynomials satisfy the identity
\begin{align}\label{cp1}
R_{n}^{-a,b,N}(\lambda^{-a,b}(x))&=R_{n}^{-a,-b,N+b}(\lambda^{-a,-b}(x+b)).
\end{align}

When $N$ is a positive integer and $a ,b \not =-1,-2,\cdots -N $, $a+b \not=-1,\cdots, -2N-1$, the dual Hahn polynomials $R_n^{a,b,N}$, $n=0,\cdots , N$, are
orthogonal with respect to the following measure
\begin{align}\label{masdh}
\rho_{a,b,N}&=\sum _{x=0}^N \frac{(2x+a+b+1)(a+1)_x(-N)_xN!}{(-1)^x(x+a+b+1)_{N+1}(b+1)_xx!}\delta_{\lambda^{a,b}(x)},
\\\label{normedh}
\langle R_n^{a,b,N},R_n^{a,b,N}\rangle &=\frac{(-N)_n^2\binom{a+n}{n}}{\binom{b+N-n}{N-n}},\quad  n=0,\cdots , N.
\end{align}
The measure $\rho_{a,b,N}$ is positive or negative only when either $-1<a,b$ or $a,b<-N$, respectively.

If $N$ is not a nonnegative integer and $a,-b-N-1\not =-1,-2,\cdots$, the dual Hahn polynomials $(R_n^{a,b,N})_n$ are always orthogonal with respect to a signed measure.

We write $(h_{n}^{a,b,N})_n$ for the sequence of Hahn polynomials defined by
\begin{equation}\label{hpol}
h_{n}^{a,b,N}(x)=\sum _{j=0}^n\frac{(-n)_j(a+b+n+1)_j(-N+j)_{n-j}(a+j+1)_{n-j}(-x)_j}{j!}.
\end{equation}
We have taken a different normalization that in \cite{dudh} since
we deal here with the case when $a$ is a negative integer (see \cite{KLS}, pp, 204-8).

When $\ax,\bx\in\{-1,-2,\cdots\}$, $\ax\le \bx$, we have that
\begin{align}\label{cos1}
&h_{n}^{\ax,\bx,N}(x)=0,\quad \mbox{for $-\ax\le n\le -\ax-\bx-1$,}\\\label{cos2}
&\mbox{$h_{n}^{\ax,\bx,N}$ has degree $-\ax-\bx-n-1$, with $-\bx\le-\ax-\bx-n-1<n$,}\\\nonumber &\hspace{4cm} \mbox{for $\lceil\frac{-\ax-\bx}{2}\rceil \le n\le -\ax-1$},\\\label{cos3}
&\mbox{$h_{n}^{\ax,\bx,N}$ has degree $n$, for $n\not \in \{\lceil\frac{-\ax-\bx}{2}\rceil,\cdots ,-\ax-\bx-1\}$,}\\\nonumber&\hspace{2cm}\mbox{and it is divisible by $(x+\ax+1)_{-\ax}$ when $n\ge -\ax-\bx$.}
\end{align}
The hypergeometric representation of the Hahn and dual Hahn polynomials
shows the following duality when $n,m\ge 0,$
\begin{equation}\label{sdm2b}
(a+1)_n(-N)_n h_{m}^{a,b,N}(n)=n!(a+1)_m(-N)_m R_{n}^{a,b,N}(\lambda^{a,b}(m)).
\end{equation}
Hahn polynomials also satisfies the following identities.
\begin{align}\label{hcp}
(-1)^nh_{n}^{a,b,N}(x)&=h_{n}^{b,a,N}(N-x),\\\label{fi1}
(-1)^nh_{n}^{a,b,N}(x)&=h_{n}^{a,b,-a-b-2-N}(-x-a-1),\\\label{cph}
h_{n+a+b}^{-a,-b,N+a+b}(x+a)&\\\nonumber&\hspace{-1.8cm}=(n+1)_{a+b}(x+1)_a(x-N-b)_bh_n^{a,b,N}(x),\quad \mbox{when $a,b\in \NN$}.
\end{align}
For $\alpha,\beta \in \RR $, we use the standard definition of the Jacobi polynomials $(P_{n}^{\alpha,\beta})_n$
\begin{equation}\label{defjac}
P_{n}^{\alpha,\beta}(x)=2^{-n}\sum _{j=0}^n \binom{n+\alpha}{j}\binom{n+\beta}{n-j}(x-1)^{n-j}(x+1)^{j}
\end{equation}
(see \cite{KLS}, pp. 216-221).

When $\ax,\bx\in\{-1,-2,\cdots\}$, $\ax\le \bx$, we have that
\begin{align}\label{cos1j}
&P_{n}^{\ax,\bx}(x)=0,\quad \mbox{for $-\ax\le n\le -\ax-\bx-1$,}\\\label{cos2j}
&\mbox{$P_{n}^{\ax,\bx}$ has degree $-\ax-\bx-n-1$, with $-\bx\le-\ax-\bx-n-1<n$,}\\\nonumber &\hspace{4cm} \mbox{for $\lceil\frac{-\ax-\bx}{2}\rceil \le n\le -\ax-1$},\\\label{cos3j}
&\mbox{$P_{n}^{\ax,\bx}$ has degree $n$, for $n\not \in \{\lceil\frac{-\ax-\bx}{2}\rceil,\cdots ,-\ax-\bx-1\}$.}
\end{align}
When $\alpha ,\beta >-1$, Jacobi polynomials are orthogonal with respect to the positive weight
\begin{equation}\label{jacw}
(1-x)^\alpha (1+x)^{\beta}, \quad -1<x<1.
\end{equation}
One can obtain Jacobi polynomials from Hahn polynomials using the limit
\begin{equation}\label{blmel}
\lim_{N\to +\infty}\frac{h_n^{\alpha,\beta,N}\left(\frac{(1-x)N}{2}\right)}{N^n}=(-1)^nn!P_n^{\alpha,\beta}(x)
\end{equation}
see \cite{KLS}, p. 207 (note that we are using for
Hahn polynomials a different normalization to that in \cite{KLS}). This limit is uniform in compact sets of $\CC$.

\section{Where do the  polynomials $\Hh _n^{\ax,\bx,N;\M}$ and $\Pp _n^{\ax,\bx;\M}$ come from?}\label{sec1}
As explained in the Introduction the polynomials $(\Hh_n^{\ax,\bx,N;\M})_n$ (Definition \ref{losw}) and $(\Pp _n^{\ax,\bx;\M})_n$ (Definition \ref{losj})  play in the new  exceptional Hahn and Jacobi families the role played by the Hahn and Jacobi polynomials in the standard families, respectively. We find the polynomials $(\Hh_n^{\ax,\bx,N;\M})_n$ in \cite{dundh}. In this Section, we show how to get the polynomials
$(\Hh_n^{\ax,\bx,N;\M})_n$ and $(\Pp _n^{\ax,\bx;\M})_n$ by taking limit in a suitable way in (\ref{xh1}) and (\ref{xj1}), respectively.

So, let $\ax, \bx$ be negative integers with $\ax\le \bx\le -1$. We
do not need to assume here that $N$ is a positive integer.
Let $F$ be a finite set of positive integers satisfying (\ref{cis}).

Write
\begin{equation}\label{losas}
\alpha_s=\ax+s/M,\quad \beta_s=\bx-s,
\end{equation}
where $s,M$ are positive real numbers with $s$ small enough so that $\alpha_s,\beta_s\not\in \ZZ$. Consider  the exceptional Hahn family defined by
\begin{equation}\label{defmexis}
h_n^{\alpha_s, \beta_s, N;F}(x)=\left|
  \begin{array}{@{}c@{}lccc@{}c@{}}
    & h_{n-u_F}^{\alpha_s ,\beta_s ,N}(x+j-1) &&\hspace{-.6cm}{}_{1\le j\le k+1} \\
    \dosfilas{ h_{f}^{\alpha_s ,\beta_s ,N}(x+j-1) }{f\in F} \\
      \end{array}
  \hspace{-.3cm}\right|.
\end{equation}
We split ut the finite set $F$ in four parts
\begin{align*}
F_p&=\{f:\mbox{$f\in F$ and $f\le \lceil\frac{-\ax-\bx}{2}\rceil -1$\},\quad $F_s=\{f:\lceil\frac{-\ax-\bx}{2}\rceil \le f\le -\ax-1\}$},\\
F_t&=\{f:-\ax\le f\le -\ax-\bx-1\},\hspace{1.2cm} F_c=\{f:\mbox{$f\in F$ and $-\ax-\bx \le f$}\}.
\end{align*}
Note that  $f\in F_s$ if and only if $-\bx\le -\ax-\bx-f-1\le \lceil\frac{-\ax-\bx}{2}\rceil -1$ and so if $f\in F_s$ then $-\ax-\bx-f-1\in F_p$ (since we assume (\ref{cis})).
Hence, except for the normalization constant $s^{n_{F_s}+n_{F_t}}$, the exceptional Hahn polynomial $h_n^{\alpha_s, \beta_s, N;F}$ (\ref{defmexis}) can be rewritten in the form
\begin{equation}\label{defmexi22}
\left|
  \begin{array}{@{}c@{}lccc@{}c@{}}
    & h_{n-u_F}^{\alpha_s ,\beta_s ,N}(x+j-1) &&\hspace{-.9cm}{}_{1\le j\le k+1} \\
    \dosfilas{ h_{f}^{\alpha_s ,\beta_s ,N}(x+j-1) }{f\in F_p} \\
    \dosfilas{ \frac{1}{s}(h_{f}^{\alpha_s ,\beta_s ,N}(x+j-1)-\frac{h_{f}^{\alpha_s ,\beta_s ,N}(\tau_f)}{h_{-\ax-\bx-f-1}^{\alpha_s ,\beta_s ,N}(\tau_f)} h_{-\ax-\bx-f-1}^{\alpha_s ,\beta_s ,N}(x+j-1)) }{f\in F_s} \\
    \dosfilas{ \frac{1}{s}h_{f}^{\alpha_s ,\beta_s ,N}(x+j-1) }{f\in F_t} \\
    \dosfilas{ h_{f}^{\alpha_s ,\beta_s ,N}(x+j-1) }{f\in F_c}
      \end{array}
  \hspace{-.3cm}\right|,
\end{equation}
where $\tau_f$ is not a root of the polynomial $h_{-\ax-\bx-f-1}^{\alpha_s ,\beta_s ,N}$.
Since
$$
F_s\cup F_t=\{f:\lceil\frac{-\ax-\bx}{2}\rceil \le f\le -\ax-\bx-1\}\subset F,
$$
we deduce from the definition of $\sigma_F$ (\ref{elsig}) that $n-u_F\not \in F_s\cup F_t$.

We next take limit  in (\ref{defmexi22}) as $s\to 0$. It is easy to see that for $f\not\in\{\lceil\frac{-\ax-\bx}{2}\rceil ,\cdots, -\ax-\bx-1\}$
(i.e. $f\not \in F_s\cup F_t$), the Hahn polynomial $h_{f}^{\alpha_s ,\beta_s ,N}(x+j-1)$
goes to $h_{f}^{\ax ,\bx ,N}(x)$ which it is a polynomial of degree $f$. This is the reason why we have defined in Definition (\ref{losw})  $\Hh_{f}^{\ax ,\bx ,N;\M}(x)=h_{f}^{\ax ,\bx ,N}(x)$ when $f\not \in \{\lceil\frac{-\ax-\bx}{2}\rceil,\cdots, -\ax-\bx-1\}$.

If $f\in F_t$, that is $-\ax\le f\le -\ax-\bx-1$, a careful computation using (\ref{hpol}) shows that except for the multiplicative constant
$M/(M-1)$, the limit
\begin{equation}\label{lim1}
\lim_{s\to 0}\frac{1}{s}h_{f}^{\ax+s/M,\bx-s,N}(x)
\end{equation}
coincides with the combination of two Hahn polynomials in the right hand side of the identity (\ref{losw2}), and this is the reason why
we have defined $\Hh_{f}^{\ax ,\bx ,N,\M}(x)$ in that form when $-\ax\le f\le -\ax-\bx-1$. Note that, in Definition \ref{losw} we have taken an arbitrary parameter $M_{\ax+f}$ for each $f$, $-\ax\le f\le -\ax-\bx-1$. Hence, we conclude
\begin{equation}\label{lim1b}
\Hh_{f}^{\ax ,\bx ,N;\M}(x)=\frac{M_{\ax+f}}{M_{\ax+f}-1}\lim_{s\to 0}\frac{1}{s}h_{f}^{\ax+s/M_{\ax+f},\bx-s,N}(x).
\end{equation}

If $f\in F_s$, that is $\lceil\frac{-\ax-\bx}{2}\rceil \le f\le -\ax-1$, and $\tau_f$ is not a root of $h_{-f-\ax-\bx-1}^{\ax,\bx,N}$, it is not difficult to see that
$$
\lim_{s\to 0}\frac{h_{f}^{\ax+s/M,\bx-s,N}(\tau_f)}{h_{-f-\ax-\bx-1}^{\ax+s/M,\bx-s,N}(\tau_f)}=\frac{h_{f}^{\ax,\bx,N}(0)}{h_{-f-\ax-\bx-1}^{\ax,\bx,N}(0)}.
$$
Hence a careful computation shows that when $\tau_f$ is not a root of $h_{-f-\ax-\bx-1}^{\ax,\bx,N}$,
the limit
$$
\lim_{s\to 0}\frac{1}{s}
\left(h_{f}^{\ax+s/M,\bx-s,N}(x)-\frac{h_{f}^{\ax+s/M,\bx-s,N}(\tau_f)}{h_{-f-\ax-\bx-1}^{\ax+s/M,\bx-s,N}(\tau_f)}h_{-f-\ax-\bx-1}^{\ax+s/M,\bx-s,N}(x)\right)
$$
is always a polynomial of degree $f$. Any of these limit polynomials would be a good candidate for defining $\Hh_{f}^{\ax ,\bx ,N;\M}$. We choose $\tau_f=0$ (which it seems to be the simplest choice). Then, it easy to see that, except for the multiplicative constant
$M/(M-1)$, the limit
$$
\lim_{s\to 0}\frac{1}{s}
\left(h_{f}^{\ax+s/M,\bx-s,N}(x)-\frac{h_{f}^{\ax+s/M,\bx-s,N}(0)}{h_{-f-\ax-\bx-1}^{\ax+s/M,\bx-s,N}(0)}h_{-f-\ax-\bx-1}^{\ax+s/M,\bx-s,N}(x)\right)
$$
coincides with the combination of derivatives of the hypergeometric function in the right hand side of the identity (\ref{losw1}), and this is the reason why we have defined $\Hh_{f}^{\ax ,\bx ,N,\M}(x)$ in this form when $\lceil\frac{-\ax-\bx}{2}\rceil \le f\le -\ax-1$.
Note that in this case, the parameter $M$ does not play any role, and moreover, we have
\begin{equation}\label{lim2h}
\Hh_{f}^{\ax ,\bx ,N;\M}(x)=\lim_{s\to 0}\frac{1}{s}
\left(h_{f}^{\ax-s,\bx,N}(x)-\frac{h_{f}^{\ax-s,\bx,N}(0)}{h_{-f-\ax-\bx-1}^{\ax-s,\bx,N}(0)}h_{-f-\ax-\bx-1}^{\ax-s,\bx,N}(x)\right).
\end{equation}
A careful computation gives the following explicit expression for the polynomial $\Hh_f^{\ax,\bx,N;\M}$ in (\ref{losw1}) when $\lceil\frac{-\ax-\bx}{2}\rceil -1\le f\le -\ax-1$:
\begin{align*}\label{elhe1}
\frac{(-f,-x)_{-\ax-\bx-f}}{(-1)^{f-\ax-\bx}}&\sum_{j=0}^{2f+\ax+\bx}(j-N-\ax-\bx-f,j-\bx-f+1)_{2f+\ax+\bx-j}\\\nonumber
&\hspace{2.5cm}\times \frac{(-2f-\ax-\bx,-x-\ax-\bx-f)_j}{(-f-\ax-\bx)\binom{j-\ax-\bx-f}{j}}\\\nonumber
&\hspace{-.5cm}+\sum_{j=0}^{-\ax-\bx-f-1}\frac{(j-N,\ax+j+1)_{f-j}(-f,f+\ax+\bx+1,-x)_j}{j!}\\\nonumber
&\hspace{2.53cm}\times \sum_{i=0}^{j-1}\frac{(2f+\ax+\bx+1)}{(-f+i)(f+\ax+\bx+1+i)}.
\end{align*}

We can now adapt this approach for the case $\bx\le\ax\le-1$.

\begin{definition}\label{loswt} Let $\ax,\bx$ be negative integers with $\bx<\ax\le -1$ and $N$ a real number.
We define the sequence $(\Hh _n^{\ax,\bx,N;\M})_n$ of polynomials, $\Hh_n^{\ax,\bx,N;\M}$ of degree $n$, as follows.

For $n\in\{\lceil\frac{-\ax-\bx}{2}\rceil ,\cdots, -\bx-1\}$,
\begin{equation}\label{loswt1}
\Hh_{n}^{\ax ,\bx ,N;\M}(x)=\lim_{s\to 0}\frac{1}{s}
\left(h_{f}^{\ax-s,\bx,N}(x)-\frac{h_{f}^{\ax-s,\bx,N}(N)}{h_{-f-\ax-\bx-1}^{\ax-s,\bx,N}(N)}h_{-f-\ax-\bx-1}^{\ax-s,\bx,N}(x)\right);
\end{equation}
for $f\in\{-\bx ,\cdots, -\ax-\bx-1\}$,
\begin{equation}\label{loswt2}
\Hh_{f}^{\ax ,\bx ,N;\M}(x)=\frac{M_{\bx+f}}{M_{\bx+f}-1}\lim_{s\to 0}\frac{1}{s}h_{f}^{\ax+s/M_{\bx+f},\bx-s,N}(x);
\end{equation}
otherwise
\begin{equation}\label{loswt3}
\Hh _n^{\ax,\bx,N;\M}(x)=h _n^{\ax,\bx,N}(x).
\end{equation}
\end{definition}
With this definition, it is easy to see that the polynomials $(\Hh _n^{\ax,\bx,N;\M})_n$ inherit the symmetry of the Hahn polynomials with respect to the interchange of the parameters $\ax$ and $\bx$:
\begin{equation}\label{fi2a}
(-1)^n\Hh_{n}^{\ax,\bx,N;\M}(x)=\Hh_{n}^{\bx,\ax,N;\M^{-1}}(N-x)
\end{equation}
where we write $\M^{-1}$ for the set of parameters $\{1/M_0,1/M_1,\cdots \}$
(to get this identity is the reason why we have substituted $0$ by $N$ in (\ref{loswt1}) with respect to (\ref{lim2h})). Using (\ref{fi2a}), we can find explicit expressions for $\Hh _n^{\ax,\bx,N;\M}$ when $\bx\le \ax$ from those ones for $\Hh _n^{\bx,\ax,N;\M}$.

From Definition \ref{losw} and the identity (\ref{fi2a}), it is not difficult to see that the polynomial $\Hh_{n}^{\ax ,\bx ,N;\M}$ has degree $n$ and leading coefficient equal to
\begin{equation}\label{lcha}
\begin{cases} (-1)^{n+\ax+\bx}(2n+\ax+\bx)!(-\ax-\bx-n-1)!,&\lceil\frac{-\ax-\bx}{2}\rceil\le n\le -\ax-\bx-1,\\
(\ax+\bx+n+1)_n,&\mbox{otherwise}.
\end{cases}
\end{equation}

\bigskip

As explained in the Introduction, we can construct the polynomials $(\Pp _n^{\ax,\bx;\M})_n$ (Definition \ref{losj}) in two different ways. On the one hand, we can proceed similarly as we have done to get the polynomials $\Hh_{n}^{\ax ,\bx ,N;\M}(x)$ but using the Jacobi polynomials in the determinant (\ref{xj1}) instead of the Hahn polynomials in the determinant (\ref{xh1}). In doing that, for $f$ satisfying $-\ax\le f\le -\ax-\bx-1$, we get
\begin{equation}\label{lim1j}
\Pp_{f}^{\ax ,\bx ;\M}(x)=\frac{M_{\ax+f}}{M_{\ax+f}-1}\lim_{s\to 0}\frac{1}{s}P_{f}^{\ax+s/M_{\ax+f},\bx-s}(x),
\end{equation}
and when $\lceil\frac{-\ax-\bx}{2}\rceil \le f\le -\ax-1$, we have
\begin{equation}\label{lim2j}
\Pp_{f}^{\ax ,\bx ;\M}(x)=\lim_{s\to 0}\frac{1}{s}
\left(P_{f}^{\ax-s,\bx}(x)-\frac{P_{f}^{\ax-s,\bx}(1)}{P_{-f-a-b-1}^{\ax-s,\bx}(1)}P_{-f-\ax-\bx-1}^{\ax-s,\bx}(x)\right).
\end{equation}
A careful computation gives the following explicit expression for the polynomial $\Pp_f^{\ax,\bx;\M}$ in (\ref{losj1}) when $\lceil\frac{-\ax-\bx}{2}\rceil -1\le f\le -\ax-1$:
\begin{align*}\label{elhe12}
&\frac{(-1)^f(x+1))^{-\ax-\bx-f}}{(2f+\ax+\bx)!}\sum_{j=0}^{2f+\ax+\bx}\frac{(-2f-\ax-\bx)_j(j-\ax-f+1)_{2f+\ax+\bx-j}(x+1)^j}{2^{-\ax-\bx-f+j}(-f-\ax-\bx)\binom{j-\ax-\bx-f}{j}}
\\\nonumber&\hspace{2.5cm}+
\frac{(-1)^f}{f!}\sum_{j=0}^{-f-\ax-\bx-1}\frac{(\bx+j+1)_{f-j}(-f,f+\ax+\bx+1)_j(x+1)^j}{2^jj!}
\\\nonumber&\hspace{2.9cm}\times\left[\sum_{i=0}^{j-1}\frac{(2f+\ax+\bx+1)}{(-f+i)(f+\ax+\bx+1+i)}
-\sum_{i=0}^{2f+\ax+\bx}\frac{1}{-\ax-f+i}\right].
\end{align*}

On the other hand, we can construct the polynomials $(\Pp _n^{\ax,\bx;\M})_n$
taking into account that the exceptional Jacobi polynomials can be constructed by setting $x\to (1-x)N/2$ in the exceptional Hahn polynomials and taking limit as $N\to +\infty$ (i.e., in the same way as Hahn polynomials produce Jacobi polynomials (\ref{blmel})). Hence, if we set $x\to (1-x)N/2$ in $\Hh_{n}^{\ax ,\bx ,N;\M}$ and take limit as $N\to +\infty$, using (\ref{blmel}) we deduce
\begin{equation}\label{blmel2}
\lim_{N\to +\infty}\frac{\Hh_n^{\ax,\bx,N;\M}\left(\frac{(1-x)N}{2}\right)}{N^n}=(-1)^nn!\Pp_n^{\ax,\bx;\M}(x)
\end{equation}
uniform in compact sets of $\CC$.

We can now adapt any of these approaches for the case $\bx\le\ax\le-1$.

\begin{definition}\label{losjt} Let $\ax,\bx$ be negative integers with $\bx<\ax\le -1$.
We define the sequence $(\Pp _n^{\ax,\bx;\M})_n$ of polynomials, $\Pp_n^{\ax,\bx;\M}$ of degree $n$, as follows.

For $n\in\{\lceil\frac{-\ax-\bx}{2}\rceil ,\cdots, -\bx-1\}$,
\begin{equation}\label{losjt1}
\Pp_{n}^{\ax ,\bx ;\M}(x)=\lim_{s\to 0}\frac{1}{s}
\left(P_{f}^{\ax-s,\bx}(x)-\frac{P_{f}^{\ax-s,\bx}(-1)}{P_{-f-\ax-\bx-1}^{\ax-s,\bx}(-1)}P_{-f-\ax-\bx-1}^{\ax-s,\bx}(x)\right);
\end{equation}
for $f\in\{-\bx ,\cdots, -\ax-\bx-1\}$,
\begin{equation}\label{losjt2}
\Pp_{f}^{\ax ,\bx ;\M}(x)=\frac{M_{\bx+f}}{M_{\bx+f}-1}\lim_{s\to 0}\frac{1}{s}P_{f}^{\ax+s/M_{\bx+f},\bx-s}(x);
\end{equation}
otherwise
\begin{equation}\label{losjt3}
\Pp _n^{\ax,\bx;\M}(x)=h _n^{\ax,\bx}(x).
\end{equation}
\end{definition}
With this definition, it is easy to see that the polynomials $(\Pp _n^{\ax,\bx;\M})_n$ inherit the symmetry of the Jacobi polynomials with respect to the interchange of the parameters $\ax$ and $\bx$:
\begin{equation}\label{fi2ja}
(-1)^n\Pp_{n}^{\ax,\bx;\M}(x)=\Pp_{n}^{\bx,\ax;\M^{-1}}(-x).
\end{equation}
Using (\ref{fi2ja}), we can find explicit expressions for $\Pp _n^{\ax,\bx;\M}$ when $\bx\le \ax$ from those ones for $\Pp _n^{\bx,\ax;\M}$.

From Definition \ref{losj} and the identity (\ref{fi2ja}), it is not difficult to see that the polynomial $\Pp_{n}^{\ax ,\bx ;\M}$ has degree $n$ and leading coefficient equal to
\begin{equation}\label{lcja}
\begin{cases} \frac{(2n+\ax+\bx)!(-\ax-\bx-n-1)!}{(-1)^{n+\ax+\bx}2^nn!},&\lceil\frac{-\ax-\bx}{2}\rceil\le n\le -\ax-\bx-1,\\
\frac{(\ax+\bx+n+1)_n}{2^nn!},&\mbox{otherwise}.
\end{cases}
\end{equation}

\bigskip
The sequences $(\Hh _n^{\ax,\bx,N;\M})_n$ and $(\Pp _n^{\ax,\bx;\M})_n$ inherit many of the structural formulas that the Hahn and Jacobi polynomials enjoy, respectively. However, there are slight perturbations in these formulas (and the perturbations cause many problems in the proofs of the results).
Here it is an instance of these formulas which we will use later on (note the perturbation in (\ref{fi1a2}) with respect to (\ref{fi1a})). The proof is a matter of calculation and is omitted.

For $n\not \in \{\lceil\frac{-\ax-\bx}{2}\rceil,\cdots , f\le -\ax-1\}$ we have
\begin{align}\label{fi1a}
\Hh_{n}^{\ax,\bx,N;\M}(x)=(-1)^n\Hh_{n}^{\ax,\bx,-\ax-\bx-2-N;\M}(-x-\ax-1);
\end{align}
and for $n\in \{\lceil\frac{-\ax-\bx}{2}\rceil,\cdots , f\le -\ax-1\}$
\begin{align}\label{fi1a2}
\Hh_{n}^{\ax,\bx,N;\M}(x)&=(-1)^n\Hh_{n}^{\ax,\bx,-\ax-\bx-2-N;\M}(-x-\ax-1)\\\nonumber &\hspace{2.2cm}+\gamma_n^{\ax,\bx,N} h_{-n-\ax-\bx-1}^{\ax,\bx,N}(x),
\end{align}
where
\begin{equation}\label{lpc}
\gamma_n^{\ax,\bx,N}=(-1)^{\ax+\bx}(-n-\bx)_{2n+\ax+\bx+1}\sum_{j=0}^{2n+\ax+\bx}\frac{(N-n+1)_{2n+\ax+\bx+1}}{N-n+i+1}.
\end{equation}


\section{New exceptional Hahn families depending on an arbitrary number of continuous parameters}\label{sec3}
As in the rest of this paper, $\M$ denotes the set of real parameters $\M=\{M_0,M_1,\cdots \}$, $F$ a finite set of positive integers, and $\ax$ and $\bx$ denote negative integers. Along this Section, $N$ denotes a real number.

\begin{definition}\label{dll}
We associate to $\ax,\bx,N$, $\M$ and $F$ the sequence of polynomials
\begin{equation}\label{defmexi2}
h_n^{\ax, \bx, N;\M, F}(x)=\left|
  \begin{array}{@{}c@{}lccc@{}c@{}}
    & \Hh_{n-u_F}^{\ax ,\bx ,N;\M}(x+j-1) &&\hspace{-.4cm}{}_{1\le j\le n_F+1} \\
    \dosfilas{\Hh_{f}^{\ax ,\bx ,N;\M}(x+j-1) }{f\in F} \\
  \end{array}
  \hspace{-.4cm}\right|,
\end{equation}
where $n\in \sigma _F$ (\ref{elsig}) and $(\Hh_n^{\ax,\bx,N;\M})_n$ are the polynomials introduced in Definitions \ref{losw} and \ref{loswt} (depending on whether $\ax\le\bx$ or $\bx\le \ax$).
\end{definition}

Using \cite[Lemma 3.4]{DdI}, we deduce that  $h_n^{\ax,\bx,N;\M,F}$, $n\in \sigma _F$, is a polynomial of degree $n$ with leading coefficient equal to
\begin{equation*}\label{lcrn}
V_{F}\prod_{i\in \{n-u_F\},F}r_i^{\ax,\bx;\M}\prod_{f\in F}(f-n+u_F),
\end{equation*}
where $V_F$ is the Vandermonde determinant (\ref{defvdm}) and $r_i^{\ax,\bx;\M}$  is the leading coefficient of the  polynomial $\Hh_i^{\ax,\bx,N;\M}$ (see (\ref{lcha})).

We only have to consider the case $\ax\le \bx$ because it follows easily from (\ref{fi2a}) that
$$
h_n^{\ax, \bx, N;\M, F}(x)=(-1)^nh_n^{\bx, \ax, N;\M^{-1}, F}(N-x-n_F).
$$
Hence, from now on, we assume $\ax\le\bx$. There are some good reasons (which we explain in Section \ref{secu}) to assume also that
\begin{equation}\label{cis2}
\{-\bx,\cdots, -\ax-\bx-1\}\subset F.
\end{equation}

The sequence of polynomials $h_n^{\ax,\bx,N;\M, F}$, $n\in \sigma _F$, only depend on the parameters
$M_i$, $i\in F_{\bx}$, where
\begin{equation}\label{conp2}
F_\bx=\{0,1,\cdots, -\bx-1\}\setminus \{-\bx-f-1:f\in F\}.
\end{equation}
Indeed, if $f\in F$, $-\ax\le f\le -\ax-\bx-1$ and $-\ax-\bx-f-1\in F$ then we can use the polynomial $\Hh_{-\ax-\bx-f-1}^{\ax,\bx,N;\M}$ in the determinant (\ref{defmexi2}) to remove the second summand in the right hand side of the identity (\ref{losw2}) which defines the polynomial $\Hh_{f}^{\ax,\bx,N;\M}$. In doing that we remove the dependence of the polynomial $h_{n}^{\ax,\bx,N;\M,F}$ on the parameter $M_{\ax+f}$. More precisely, enumerate the polynomials $h_n^{\ax,\bx,N;\M, F}$, $n\in \sigma _F$, in accordance to the position of $n$ in the set $\sigma _F$ (i.e., the first polynomial would be $h_{u_F}^{\ax,\bx,N;\M, F}$)
and similarly enumerate the parameters $M_i$, $i\in F_{\bx}$, in accordance to the position of $i$ in the set $F_b$. it is then not difficult to check that for $i=1,\cdots, n_{F_b}$, the $i$-th polynomial $h_{n}^{\ax,\bx,N;\M,F}$ does not depend on the $(n_{F_b}-i)$-th parameter, and for $i\ge n_{F_b}+1$, the $i$-th  polynomial $h_{n}^{\ax,\bx,N;\M,F}$  depends on all the parameters $M_i$, $i\in F_b$.

The following property will be useful to show that the exceptional Legendre polynomials introduced in \cite{xle} are particular cases of the exceptional Jacobi polynomials introduced here.

\begin{remark}\label{hlp1} We first renormalize the polynomials $h_{n}^{\ax,\bx,N;\M,F}$ as follows:
$$
\bar h_{n}^{\ax,\bx,N;\M,F}=(\prod_{i=0}^{-\bx -1}(M_i-1))h_{n}^{\ax,\bx,N;\M,F},\quad n\in \sigma_F.
$$
For a finite set $J$ of nonnegative integers, we denote by $\M_J$ the particular case of the set of parameters $\M$ obtained by setting $M_j=1$, $j\in J$.

If $f\in\{-\ax,\cdots,-\ax-\bx-1\}\cap F$ and $-f-\ax-\bx-1\not \in F$, write $\tilde F=(F\setminus\{f\})\cup \{-f-\ax-\bx-1\}$ (i.e., we remove from $F$ the positive integer $f$ and include $-f-\ax-\bx-1$). Then
for $n\in \sigma_F$, $n\not =u_F-f-\ax-\bx-1$
$$
\bar h_{n-(2f+\ax+\bx+1)}^{\ax, \bx, N;\M, \tilde F}=\frac{1-M_{\ax+f}}{(-1)^{n_f}c_f}\bar h_n^{\ax, \bx, N;\M_{\{\ax+f\}}, F},
$$
where
$$
c_f=(-1)^{\bx+f}(f+\bx)!(f+\ax)!(-N-\ax-\bx-f-1)_{2f+\ax+\bx+1},
$$
and $n_f$ denotes the number of elements in $F$ which are bigger than  $-f-\ax-\bx-1$ and less than $f$;
similarly
$$
h_{u_F-f-\ax-\bx-1}^{\ax, \bx, N;\M, \tilde F}=(-1)^{n_f-1}h_{u_F-f-\ax-\bx-1}^{\ax, \bx, N;\bar \M_{\{\ax+f\}}, F}.
$$
Indeed, if $f\in\{-\ax,\cdots,-\ax-\bx-1\}\cap F$ then $0\le-f-\ax-\bx-1\le -\bx-1$. (\ref{losw2}) then gives
$\left[(M_{\ax+f}-1)\Hh_f^{\ax,\bx,N;\M}\right]_{\vert M_{\ax+f}=1}=c_f\Hh_{-f-\ax-\bx-1}^{\ax,\bx,N;\M}$, from where the remark follows easily.
\end{remark}

\subsection{Duality with Krall dual Hahn families}
As for the standard families, the tool we use to prove that the polynomials introduce in Definition \ref{dll} are exceptional Hahn polynomials is the duality of these polynomials
and certain  Krall dual Hahn polynomials. In this case, we consider the Krall dual Hahn polynomials constructed in \cite{dundh} which we next display.
We do not need to assume yet that $N$ is a positive integer but assume that $\ax\le\bx\le-1$ and that the finite set $F$ satisfies (\ref{cis}).

Using again the notation (\ref{mn}) (i.e. $a=-\ax,b=-\bx$ and $\Nn=N+\ax+\bx$) we set
\begin{equation}\label{losww}
W_n^{a,b,\Nn;\M}(x)=\Hh_n^{\ax,\bx,-2-\Nn;\M}(x),
\end{equation}
where $(\Hh_n^{\ax,\bx,-2-\Nn;\M})_n$ is the sequence of polynomials introduced in Definition \ref{losw}.
Consider the finite set of integers $U_F$ (\ref{elu})  and define the polynomials
\begin{equation}\label{qusmei2u}
q_n^{a,b,\Nn;\M,U_F}(x)=
\frac{\left|
  \begin{array}{@{}c@{}lccc@{}c@{}}
  & (-1)^{j-1}R_{n-a+j-1}^{a,b,\Nn}(x) &&\hspace{-2.6cm}{}_{1\le j\le n_F+1} \\
    \dosfilas{(-1)^{j-1}R_{n-a+j-1}^{a,b,\Nn}(\lambda^{a,b}(u))}{u\in U_{F_-}}\\
    \dosfilas{(a+b+\Nn-n-j+2)_{j-1}W_{f}^{a,b,\Nn;\M}(-n+a-j) }{f\in \{b,b+1,\cdots, a+b-1\}}\\
    \dosfilas{(-1)^{j-1}R_{n-a+j-1}^{a,b,\Nn}(\lambda^{a,b}(u))}{u\in U_{F_+}}
  \end{array}\hspace{-.5cm}\right|}{(-1)^{(n+a+b)(n_{U_F}+1)+\binom{a+b}{2}+\binom{b}{2}+a}\prod_{u\in U_F}(x-\lambda^{a,b}(u))}.
\end{equation}

Under mild conditions on the parameters, we prove in \cite{dundh} that the polynomials $q_n^{a,b,\Nn;\M,U_F}(x)$, $n\ge 0$, are orthogonal with respect to the measure $\nu_{a,b,\Nn}^{\M,U_F}$ (see (\ref{lctnu})). In particular, that is the case
when $N$ is a positive integer and the measure $\nu_{a,b,\Nn}^{\M,U_F}$ is positive; the polynomials (\ref{qusmei2u}) have then positive norm when $n=0,\cdots, N-n_F$. Although it is not important for the construction of the exceptional Hahn polynomials,
we also prove in \cite{dundh} that these polynomials (\ref{qusmei2u}) are eigenfunctions of a higher order difference operator.
We point out that in the determinant which appears in the right hand side of (\ref{qusmei2u}), we have rearranged rows and columns and renormalized with a sign with respect to the definition of the polynomials $q_n^{a,b,\Nn;\M,U_F}(x)$ in \cite{dundh}.

We next prove that the  polynomials $h_n^{\ax,\bx,N;\M,F}$ (\ref{defmexi}) are related by duality with the polynomials $q_n^{a,b,\Nn;\M,U_F}$.

\begin{lemma}\label{lem3.2}
If $n\ge 0$ and $v\in \NN\setminus F$, then
\begin{align}\label{duaqnrn}
\xi_{n}&h_{v+u_F}^{\ax,\bx,N;\M , F}(n)=\kappa_\M\tau_v\zeta_{v+\ax+\bx}\theta_{v}^{\M}q_{n}^{a,b,\Nn;\M,U_F}(\lambda^{a,b}(v+\ax+\bx)),
\end{align}
where
\begin{align*}
\xi_n&=(N-n+1)_{n+\ax+\bx}^{\ax+n_F+1}\prod_{j=1}^{n_F+1}(N-n-j+2)_{j-1},\\
\zeta_v&=(-N-\ax-\bx)_v(v+1)_{-\ax}(v-\ax-\bx)!,\\
\kappa_\M&=\left(\prod_{u\in U_{F_-}}(1-M_{-u+\ax-1})\right)\left(\prod_{u\in U_F}\zeta_u\right),\\
\tau_v&=\prod_{u\in U_F}(\lambda^{-\ax,-\bx}(v+\ax+\bx)-\lambda^{-\ax,-\bx}(u)),\\
\theta_{v}^{\M}&= \begin{cases}1-M_{-\bx-1-v},&0\le v\le-\bx-1,\\ 1,&\mbox{otherwise}.\end{cases}
\end{align*}
\end{lemma}

\begin{proof}
We  dualize each  entry $(i,j)$, $i,j=1\cdots, n_F+1$, of (the determinant which defines) the polynomial $h_{v+u_F}^{\ax,\bx,N;\M ,F}(n)$ (\ref{defmexi})
and compare with the entry $(i,j)$ of (the determinant which defines) the polynomial $q_{n}^{a,b,\Nn;\M,U_F}(\lambda^{a,b}(v+\ax+\bx))$ (\ref{qusmei2u}).

We proceed in several steps, depending on the rows and on $v\in \NN\setminus F$.

\noindent
\textit{First step.} Consider the first row $i=1$ and assume that $v\ge -\ax-\bx$.
From (\ref{defmexi}) and Definition \ref{losw}, we deduce that the entry $(1,j)$, $j=1,\cdots , n_F+1$, of $h_{v+u_F}^{\ax,\bx,N;\M,F}(n)$ has the form $h_{v}^{\ax,\bx,N}(n+j-1)$. We then prove that
\begin{align}\label{ptg}
&(N-n+1)_{n+\ax+\bx}(N-n-j+2)_{j-1}h_{v}^{\ax,\bx,N}(n+j-1)\\\nonumber
&\hspace{1cm}=
(v+\ax+\bx+1)_{-\ax}v!(-N-\ax-\bx)_{v+\ax+\bx}\\\nonumber&\hspace{2cm}\times (-1)^{n+j-1+\ax+\bx} R^{a,b,\Nn}_{n-a+j-1}(\lambda^{a,b}(v+\ax+\bx)).
\end{align}

Note that $R^{a,b,\Nn}_{n-a+j-1}(\lambda^{a,b}(v+\ax+\bx))$ is the  entry $(1,j)$  of the determinant which defines the polynomial $q_{n}^{a,b,\Nn;\M,U_F}(\lambda^{a,b}(v+\ax+\bx))$ (\ref{qusmei2u}).
Indeed, if $n+j-1\ge -\ax$, by applying to $h_{v}^{\ax,\bx,N}(n+j-1)$ firstly the identity (\ref{cph}) and then the duality (\ref{sdm2b}) (to do that we need $n+j-1\ge -\ax$), we get (\ref{ptg}) after straightforward computations. For $n+j-1\le -\ax-1$, the identity (\ref{ptg}) holds because both sides are equal to zero: the left hand side because for $v\ge -\ax-\bx$, $h_{v}^{\ax,\bx,N}(n+j-1)=0$ for $n+j-1=0,\cdots,-\ax-1$ (see (\ref{cos3})), and the right hand side because $n-a+j-1\le -1$ and then $R^{a,b,\Nn}_{n-a+j-1}=0$.

\medskip

\noindent
\textit{Second step.} Consider the first row $i=1$ and assume now that $v\le -\ax-\bx-1$.
Since $v\not \in F$ (see (\ref{elsig})), we have $0\le v\le -\bx-1$. If write $g=-\ax-\bx-1-v$, then $-\ax\le g\le -\ax-\bx-1$. In view of (\ref{losw2}), we see that the polynomial $h_{v}^{\ax,\bx,N}(x)$ (which defines the first row $i=1$ of $h_{v+u_F}^{\ax,\bx,N;\M,F}(n)$) is the polynomial in the first summand of $\Hh_{g}^{\ax,\bx,N;\M}(x)$. Since $\Hh_{g}^{\ax,\bx,N;\M}(x)$ in turn defines the $(2+n_{U_{F_-}}+g+\bx)$-th row of the determinant (\ref{defmexi}) (from which  $h_{v+u_F}^{\ax,\bx,N;\M,F}(n)$ is defined), the polynomial $h_{v+u_F}^{\ax,\bx,N;\M,F}(n)$ remains the same if we add the $(2+n_{U_{F_-}}+g+\bx)$-th row of the determinant (\ref{defmexi}) multiplied by
$$
\frac{(M_{\ax+g}-1)(-1)^{\bx+g+1}}{(g+\ax)!(g+\bx)!(-N-\ax-\bx-g-1)_{2g+\ax+\bx+1}}
$$
to the first row of the determinant (\ref{defmexi}). In doing that, we deduce  from (\ref{losw2}) that the entry $(1,j)$ in the first row of $h_{v+u_F}^{\ax,\bx,N;\M,F}(n)$ can be taken to be
$$
\tilde h_{v,n}^{1,j}=\frac{(1-M_{\ax+g})(-\ax-\bx-g-1)!(-n-j+1)_{-\ax}}{(g+\ax)!(-N-\ax-\bx-g-1)_{2g+\ax+\bx+1}}h_{\ax+g}^{-\ax,\bx,N+\ax}(n+j-1+\ax),
$$
$j=1,\cdots, n_F+1$. If $n+j-1+\ax\ge 0$, by applying to $h_{\ax+g}^{-\ax,\bx,N+\ax}(n+j-1+\ax)$  the duality (\ref{sdm2b}) and then the identity (\ref{cp1}), we get after careful computations
\begin{align}\label{ptg1}
&(-1)^{n+j-1+\ax+\bx}(N-n+1)_{n+\ax+\bx}(N-n-j+2)_{j-1}\tilde h_{v,n}^{1,j}\\\nonumber
&\hspace{1.5cm}=
(v+\ax+\bx+1)_{-\ax}v!(-N-\ax-\bx)_{v+\ax+\bx}\\\nonumber&\hspace{3cm}\times (1-M_{-\bx-1-v})R^{a,b,\Nn}_{n-a+j-1}(\lambda^{a,b}(v+\ax+\bx)).
\end{align}
For $n+j-1\le -\ax-1$, the identity (\ref{ptg1}) holds because both sides are equal to zero: the left hand side because the factor
$(-n-j+1)_{-\ax}=0$ in $\tilde h_{v,n}^{1,j}$, and the right hand side because $n-a+j-1\le -1$ and then $R^{a,b,\Nn}_{n-a+j-1}=0$.

\medskip

\noindent
\textit{Third step.} Consider next the rows $i=2,\cdots, 1+n_{U_{F_-}}$. The entry $(i,j)$, $j=1,\cdots , n_F+1$, of $h_{v+u_F}^{\ax,\bx,N;\M,F}(n)$ has now the form $h_{f}^{\ax,\bx,N}(n+j-1)$, $f\in F, 1\le f\le -\bx-1$ (see (\ref{elu1})). Hence $u=f+\ax+\bx\in U_{F_-}$. Proceeding as in the second step (here $f$ plays the role of $v$), we conclude that the entry $(i,j)$ in the $i$-th row of $h_v^{\ax,\bx,N;\M,F}(n)$ can be taken to be
$$
\tilde h_{v,n}^{1,j}=\frac{(1-M_{-u+\ax-1})f!(-n-j+1)_{-\ax}}{(-\bx-f-1)!(-N+f)_{-\ax-\bx-2f-2}}h_{-\bx-f-1}^{-\ax,\bx,N+\ax}(n+j-1+\ax),
$$
$j=1,\cdots, n_F+1$, and
\begin{align}\label{ptg2}
&(-1)^{n+j-1+\ax+\bx}(N-n+1)_{n+\ax+\bx}(N-n-j+2)_{j-1}\tilde h_{v,n}^{i,j}\\\nonumber
&\hspace{1.5cm}=
(u+1)_{-\ax}(u-\ax-\bx)!(-N-\ax-\bx)_{u}\\\nonumber&\hspace{3cm}\times (1-M_{-u+\ax-1})R^{a,b,\Nn}_{n-a+j-1}(\lambda^{a,b}(u)).
\end{align}

\medskip

\noindent
\textit{Fourth step.} Consider next the rows $i=1+n_{F_-}+r$, $r=1,\cdots,-\ax$.
From (\ref{defmexi}) and Definition \ref{losw}, we deduce that the entry $(i,j)$, $j=1,\cdots , n_F+1$, of $h_{v+u_F}^{\ax,\bx,N;\M,F}(n)$ has the form
$\Hh_f^{\ax,\bx,N;\M}(n+j-1)$, $f=-\bx,\cdots ,-\ax-\bx-1$. If $f\not \in \{\lceil\frac{-\ax-\bx}{2}\rceil,\cdots, -\ax-1\}$, using
(\ref{losww}) and then (\ref{fi1a}) we have
\begin{align}\label{ptg3}
&(N-n-j+2)_{j-1}\Hh_f^{\ax,\bx,N;\M}(n+j-1)\\\nonumber &\hspace{2cm}=(-1)^f(N-n-j+2)_{j-1}W_f^{a,b,\Nn;\M}(-n+a-j).
\end{align}
Note that $(N-n-j+2)_{j-1}W_f^{a,b,\Nn;\M}(-n+a-j)$ is the  entry $(i,j)$  of the determinant which defines the polynomial $q_{n}^{a,b,\Nn;\M,U_F}(\lambda^{a,b}(v+\ax+\bx))$ (\ref{qusmei2u}).

If $f\in \{\lceil\frac{-\ax-\bx}{2}\rceil,\cdots, -\ax-1\}$, then $g=-f-\ax-\bx-1\in \{-\bx,\cdots,\lceil\frac{-\ax-\bx}{2}\rceil-1\}$, and
since $\Hh_{g}^{\ax,\bx,N;\M}(n)=h_g^{\ax,\bx,N}(n)$ in turn defines the $(2+n_{U_{F_-}}+g+\bx)$-th row of the determinant (\ref{defmexi}) (from which  $h_{v+u_F}^{\ax,\bx,N;\M,F}(n)$ is defined), the polynomial $h_{v+u_F}^{\ax,\bx,N;\M,F}(n)$ remains the same if we add the $(2+n_{U_{F_-}}+g+\bx)$-th row of the determinant (\ref{defmexi}) multiplied by $-\gamma_{f}^{\ax,\bx,N}$ (\ref{lpc}) to the $(2+n_{U_{F_-}}+f+\bx)$-th row of the determinant (\ref{defmexi}). In doing that, we deduce   that the entries of the $(2+n_{U_{F_-}}+f+\bx)$-th row of $h_{v+u_F}^{\ax,\bx,N;\M,F}(n)$ can be taken to be
$$
\Hh_{f}^{\ax,\bx,N;\M}(n+j-1)-\gamma_{f}^{\ax,\bx,N}h_{-f-\ax-\bx-1}^{\ax,\bx,N}(n+j-1).
$$
Using (\ref{losww}) and then (\ref{fi1a2}) we have
\begin{align}\label{ptg3i}
&(N-n-j+2)_{j-1}(\Hh_{f}^{\ax,\bx,N;\M}(n+j-1)-\gamma_{f}^{\ax,\bx,N}h_{-f-\ax-\bx-1}^{\ax,\bx,N}(n+j-1))\\\nonumber &\hspace{2cm}=(-1)^f(N-n-j+2)_{j-1}W_f^{a,b,\Nn;\M}(-n+a-j).
\end{align}

\medskip

\noindent
\textit{Fifth step.} Consider finally the rows $i=2-\ax+n_{U_{F_-}},\cdots, n_{F}+1$.
The entry $(i,j)$, $j=1,\cdots , n_F+1$, of $h_{v+u_F}^{\ax,\bx,N;\M,F}(n)$ has the form $h_{f}^{\ax,\bx,N}(n+j-1)$, for $f\in F$ and $f\ge -\ax-\bx$ (see (\ref{elu2})). Hence $u=f+\ax+\bx\in U_{F_+}$. Proceeding as in the first step (here $f$ plays the role of $v$), we have
\begin{align}\label{ptg4}
&(-1)^{n+j-1+\ax+\bx}(N-n+1)_{n+\ax+\bx}(N-n-j+2)_{j-1}h_{f}^{\ax,\bx,N}(n+j-1)\\\nonumber&\hspace{1.5cm} =
(u+1)_{-\ax}(u-\ax-\bx)!(-N-\ax-\bx)_{u}R^{a,b,\Nn}_{n-a+j-1}(\lambda^{a,b}(u)).
\end{align}
\bigskip

We can now prove the duality (\ref{duaqnrn}) from the identities (\ref{ptg}), (\ref{ptg1}), (\ref{ptg2}), (\ref{ptg3}), (\ref{ptg3i}) and (\ref{ptg4}).

\end{proof}

\subsection{The second order difference operator}
We next use the duality stated in Lemma \ref{lem3.2} to construct a second order difference operator with respect to which the
polynomials $h_{n}^{\ax,\bx,N;\M,F}(x)$, $n\in \sigma_F$, are eigenfunctions. The second order difference operator is constructed from the polynomials
$\Omega _{\M,F}^{\ax, \bx, N}$ defined in (\ref{defom}) and $\Lambda _{\M,F}^{\ax,\bx,N}(x)$ defined by
\begin{align}\label{deflam}
\Lambda _{\M,F}^{\ax, \bx, N}(x)&= \left|
  \begin{array}{@{}c@{}lccc@{}c@{}}
    & &&\hspace{-1.3cm}{}_{1\le j\le n_F+1,j\not =n_F} \\
    \dosfilas{ \Hh_{f}^{\ax ,\bx ,N}(x+j-1) }{f\in F}
  \end{array}
  \hspace{-.4cm}\right|.
\end{align}
Both are polynomials of degree $n_F+u_F$. The leading coefficient of $\Omega _{\M,F}^{\ax, \bx, N}$ is
\begin{equation}\label{lcrnw}
V_{F}\prod_{i\in F}r_i^{\ax,\bx;\M},
\end{equation}
where $V_F$ is the Vandermonde determinant (\ref{defvdm}) and $r_i^{\ax,\bx;\M}$  is the leading coefficient of the  polynomial $\Hh_i^{\ax,\bx,N;\M}$ (see (\ref{lcha})).

We need some more definitions.
We also define the sequences
\begin{align}\label{lasph}
\Phi_{n;\M,U_F}^{a,b,\Nn}&=
\frac{\left|
  \begin{array}{@{}c@{}lccc@{}c@{}}
  & &&\hspace{-2.6cm}{}_{1\le j\le n_F} \\
    \dosfilas{(-1)^{j-1}R_{n-a+j-1}^{a,b,\Nn}(\lambda^{a,b}(u))}{u\in U_{F_-}}\\
    \dosfilas{(a+b+\Nn-n-j+2)_{j-1}W_{f}^{a,b,\Nn;\M}(-n+a-j) }{f\in \{b,b+1,\cdots, a+b-1\}}\\
    \dosfilas{(-1)^{j-1}R_{n-a+j-1}^{a,b,\Nn}(\lambda^{a,b}(u))}{u\in U_{F_+}}
  \end{array}\hspace{-.6cm}\right|}{(-1)^{(n+a+b)(n_{U_F}+1)+\binom{a+b}{2}+\binom{b}{2}+a}},\\\label{lasps}
  \Psi_{n;\M,U_F}^{a,b,\Nn}&=
\frac{\left|
  \begin{array}{@{}c@{}lccc@{}c@{}}
  & &&\hspace{-2.9cm}{}_{1\le j\le n_F+1; j\not=n_F} \\
    \dosfilas{(-1)^{j-1}R_{n-a+j-1}^{a,b,\Nn}(\lambda^{a,b}(u))}{u\in U_{F_-}}\\
    \dosfilas{(a+b+\Nn-n-j+2)_{j-1}W_{f}^{a,b,\Nn;\M}(-n+a-j) }{f\in \{b,b+1,\cdots, a+b-1\}}\\
    \dosfilas{(-1)^{j-1}R_{n-a+j-1}^{a,b,\Nn}(\lambda^{a,b}(u))}{u\in U_{F_+}}
  \end{array}\hspace{-.65cm}\right|}{(-1)^{(n+a+b)(n_{U_F}+1)+\binom{a+b}{2}+\binom{b}{2}+a}}.
\end{align}

From Lemma \ref{lem3.2}, we can deduce the duality between the polynomials $\Omega _{\M,F}^{\ax,\bx,N}$ (\ref{defom}), $\Lambda _{\M,F}^{\ax,\bx,N}$ (\ref{deflam})
and the sequences (\ref{lasph}) and (\ref{lasps}), respectively:

\begin{align}\label{duomph}
\xi _n\Omega _{\M,F}^{\ax,\bx,N}(n)&=(-1)^{n+\bx}(N-n-n_F+1)_{n+\ax+\bx+n_F}\kappa_\M\Phi_n^{a,b,\Nn;\M,U_F} , \\\label{dulaps}
\xi _n\Lambda _{\M,F}^{\ax,\bx,N}(n)&=(-1)^{n+\bx}(N-n-n_F+2)_{n+\ax+\bx+n_F-1}\kappa_\M\Psi_n^{a,b,\Nn;\M,U_F}.
\end{align}

\begin{theorem}\label{th3.3} The polynomials $h_n^{\ax,\bx,N;\M,F}$ (\ref{defmexi}), $n\in \sigma _F$, are common eigenfunctions of the second order difference operator
\begin{equation}\label{sodomex}
D=h_{-1}(x)\Sh_{-1}+h_0(x)\Sh_0+h_1(x)\Sh_{1},
\end{equation}
where
\begin{align*}
h_{-1}(x)&=\frac{x(x-\bx-N-1)\Omega _{\M,F}^{\ax,\bx,N}(x+1)}{\Omega _{\M,F}^{\ax,\bx,N}(x)},\\
h_0(x)&=-(x+n_F)(x-\bx-N-1+n_F)-(x+\ax+1+n_F)(x-N+n_F)\\\nonumber &\quad \quad +\Delta\left(\frac{(x+\ax+n_F)(x-N-1+n_F)\Lambda _{\M,F}^{\ax, \bx, N}(x)}{\Omega _{\M,F}^{\ax,\bx,N}(x)}\right),\\
h_1(x)&=\frac{(x+\ax+n_F+1)(x-N+n_F)\Omega _{\M,F}^{\ax,\bx,N}(x)}{\Omega _{\M,F}^{\ax,\bx,N}(x+1)},
\end{align*}
and $\Delta $ denotes the first order difference operator $\Delta f=f(x+1)-f(x)$. Moreover $D(h_n^{\ax,\bx,N;\M,F})=\lambda^{\ax,\bx}(n-u_F)h_n^{\ax,\bx,N;\M,F}$, $n\in \sigma_F$.
\end{theorem}

\begin{proof}
The proof is similar to that of Theorem 3.3 in \cite{duch} but using here the three term recurrence relation for the polynomials
$(q_n^{a,b,\hat N;\M,U_F})_n$ in \cite[Corollary 5.2]{dundh}
and the dualities in Lemma \ref{lem3.2}, (\ref{duomph}) and (\ref{dulaps}).
\end{proof}

\subsection{Orthogonality of the polynomials $h_n^{\ax,\bx,N;\M,F}$, $n\in\sigma_F$}
In this Section we assume that $N$ is a positive integer, and define
\begin{equation}\label{sigN}
\sigma_{N;F}=\{n\in \sigma_F: n\le N+u_F\},
\end{equation}
where the set of nonnegative integers $\sigma_F$ and the nonnegative integer $u_F$ are defined in (\ref{elsig}).

As we point out in the Introduction, the key concept for the existence of a positive measure with respect to which the polynomials $(h_n^{\ax,\bx,N;\M,F})_n$ are orthogonal is that of admissibility (see Definition \ref{laad}).
This admissibility  arise from the positivity of the measure $\nu_{a,b,\Nn}^{\M, F}$ (see (\ref{mn}) and (\ref{lctnu})), but it can also be characterized by the sign of the polynomial $\Omega _{\M,F} ^{\ax,\bx, N}(x)$ when $x\in \{0,\cdots, N-n_F+1\}$.

\begin{lemma}\label{l3.1} Given two negative integers $\ax,\bx$ and a positive integer $N$,
satisfying $-N\le\ax\le \bx\le -1$, and a finite set $F$ such that (\ref{cis}) holds, the following conditions are equivalent (we use again the notation (\ref{mn}), i.e. $a=-\ax,b=-\bx$ and $\Nn=N+\ax+\bx$).
\begin{enumerate}
\item The measure $\nu_{a,b,\Nn}^{\M, F}$ is positive.
\item $\ax, \bx$, $\M$ and $F$ are  admissible.
\item $\Omega_{\M,F} ^{\ax,\bx, N}(n)\Omega_{\M,F} ^{\ax,\bx, N}(n+1)$ is positive for $n=0,\cdots, N-n_F$, where the polynomial $\Omega_{\M,F}^{\ax,\bx , N}$ is defined by (\ref{defom}).
\end{enumerate}
\end{lemma}

\begin{proof}

Note that
\begin{equation}\label{pmc}
\lambda^{a,b}(u)-\lambda^{a,b}(v)=(u-v)(u+v+a+b+1).
\end{equation}
The equivalence between parts 1 and 2 is now an easy consequence of Definition \ref{laad} (admissibility), the definition of the measures (\ref{lctnu}) and (\ref{lanu}) and the assumptions on the parameters $a,b,\Nn$.

According to \cite[Theorem 5.1]{dundh}, the norm of the polynomials $q_n^{a,b,\Nn ;\M,U_F}$ with respect to the measure $\nu_{a,b,\Nn}^{\M, U_F}$ is given by
\begin{align}\label{normqu}
&\langle q_n^{a,b,\Nn ;\M,U_F},q_n^{a,b,\Nn ;\M,U_F}\rangle_{\nu_{a,b,\Nn}^{\M,U_F}}\\\nonumber
&\quad =\frac{-n!(N+b)!^2(N+a+b-n)^{a}}
{(n+n_{U_F})!(N+a-n)!(N+a+b-n)!}
\Phi_{n;\M,U_F}^{a,b,\Nn }\Phi_{n+1;\M,U_F}^{a,b,\Nn }.
\end{align}
Using the dualities (\ref{duaqnrn}) in Lemma \ref{lem3.2} and (\ref{duomph}), we get
\begin{align}\label{norkdh}
&\langle q_n^{a,b,\Nn ;\M,U_F},q_n^{a,b,\Nn ;\M,U_F}\rangle_{\nu_{a,b,\Nn}^{\M,U_F}}=\frac{n!(N+b)!^2(N+a+b-n)^{a}}
{(n+n_{U_F})!(N+a-n)!(N+a+b-n)!}\\\nonumber
&\hspace{3cm}\times\frac{\xi_n\xi_{n+1}\Omega_{\M,F}^{\ax,\bx,N}(n)\Omega_{\M,F}^{\ax,\bx,N}(n+1)}{\kappa_\M^{2}(N-n-n_F)(N-n-n_F+1)_{n+\ax+\bx+n_F}^2}.
\end{align}
From where we deduce that
$$
\sign\left(\langle q_n^{a,b,\Nn ;\M,U_F},q_n^{a,b,\Nn ;\M,U_F}\rangle_{\nu_{a,b,\Nn}^{\M,U_F}}\right)=
\sign\left(\Omega_{\M,F}^{a,b,N}(n)\Omega_{\M,F}^{a,b,N}(n+1)\right).
$$
Part 1 $\Rightarrow$ part 2 is then an easy consequence of the positivity of the measure $\nu_{a,b,\Nn}^{\M, U_F}$. And part 2 $\Rightarrow$ part 1 follows from the part 1 of Lemma \ref{ldmp}.

\end{proof}

In the following Theorem we prove that when  $\ax,\bx $, $\M$ and $F$ are Hahn admissible the polynomials $h_n^{\ax,\bx,N;\M,F}$, $n\in \sigma _{N;F}$, are orthogonal and complete with respect to a positive measure.

\begin{theorem}\label{th4.5} Let $\ax,\bx$ and  $N$ be two negative integers and a positive integer
satisfying $-N\le\ax\le \bx\le -1$, and let $F$ be a finite set of positive integers such that (\ref{cis}) holds.
Assume that $\ax,\bx$, $\M$ and $F$ are admissible, then
the  polynomials $h_n^{\ax,\bx, N;\M, F}$, $n\in \sigma _{N;F}$, are orthogonal and complete with respect to the positive measure
\begin{equation}\label{momex}
\omega_{\ax,\bx,N}^F=\sum_{x=0}^{N-n_F} \frac{\binom{\ax +n_F+x}{x}\binom{\bx +N-x}{N-n_F-x}}{\Omega_{\M,F}^{\ax,\bx,N}(x)\Omega_{\M,F}^{\ax,\bx,N}(x+1)}\delta_x.
\end{equation}
Hence $h_n^{\ax,\bx, N;\M,F}$, $n\in \sigma _{N;F}$, are exceptional Hahn polynomials. Moreover, if we set $a=-\ax, b=-\bx, \Nn=N+\ax+\bx$, we have for $n\in\NN\setminus F$ and $n\le N$
\begin{align}\label{nomh}
\Vert h_{n+u_F}^{\ax,\bx, N;\M,F}\Vert_2=\frac{\tau_n\zeta_{n+\ax+\bx}^2(\theta_n^\M)^2\prod_{j=1}^{-\bx}(N+\ax+\bx+j)^2}
{n_{U_F}!(-\ax+\bx+n_{U_F})!\nu_{a,b,\Nn}^{\M}(n+\ax+\bx)},
\end{align}
where $\zeta_v$, $\tau_v$ and $\theta_v^\M$ are defined in Lemma \ref{lem3.2}, and  we denote by $\nu_{a,b,\Nn}^{\M}(s)$ the mass of the discrete positive measure $\nu_{a,b,\Nn}^{\M}$ (\ref{lanu}) at the point $\lambda^{a,b}(s)$ (see also (\ref{mn})).
\end{theorem}

\begin{proof}
First of all, the part 3 of Lemma \ref{l3.1} shows that the measure $\omega_{\ax,\bx, N}^{\M,F}$ is  positive.

The measure $\nu_{a,b,\Nn}^{\M,U_F}$ (\ref{lctnu}) is also positive  (part 1 of Lemma \ref{l3.1}), and it is not difficult to see that it is supported in the finite set
\begin{equation}\label{sopn}
\supp_{\nu_{a,b,\Nn}^{\M,U_F}}=\{\lambda^{a,b}(v+\ax+\bx):v\in \NN\setminus F, v\le N\},
\end{equation}
formed by $N-n_F+1$ point.

Hence, the polynomials $q_n^{a,b,\hat N;\M,U_F}$ (see (\ref{qusmei2u})), $n=0,\cdots, N-n_F$, have degree $n$ and positive $L^2$-norm. Using part 2 of Lemma \ref{ldmp}, we deduce  that the finite sequence $q_n/\Vert q_n \Vert_2$, $n=0,\cdots, N-n_F$, is an orthonormal basis in $L^2(\nu_{a,b,\Nn}^{\M,U_F})$ (to simplify the notation we remove some of the parameters, and write $q_n$ instead of $q_n^{a,b,\hat N;\M,U_F}$).

The nonnegative integers in $\sigma _{N;F}$ has the form $v+u_F$, $v\in\NN\setminus F$ and $v\le N$. For such $v$, write $s=v+\ax+\bx$. (\ref{sopn}) says that $\lambda^{a,b}(s)$ is in the support of $\nu_{a,b,\Nn}^{\M,U_F}$. Consider then the function
$$
\phi_v(x)=\begin{cases} 1/\nu_{a,b,\Nn}^{\M,U_F}(s),& x=\lambda^{a,b}(s),\\ 0,& x\not =\lambda^{a,b}(s), \end{cases},
$$
where as before we denote by $\nu_{a,b,\Nn}^{\M,U_F}(s)$ the mass of the discrete positive measure $\nu_{a,b,\Nn}^{\M,U_F}$  at the point $\lambda^{a,b}(s)$.

The function $\phi_v\in L^2(\nu_{a,b,\Nn}^{\M,U_F})$ and its Fourier coefficients with respect to the orthonormal basis $(q_n/\Vert q_n\Vert _2)_n$ are $q_n(\lambda^{a,b}(s))/\Vert q_n\Vert _2$, $n=0,\cdots, N-n_F$. Hence if we take other nonnegative integer in $\sigma_{N;F}$, that is, a number of the form $\tilde v+u_F$, $\tilde v\in\NN\setminus F$ and $\tilde v\le N$, we have that $\tilde s=\tilde v+\ax+\bx$ is in the support of $\nu_{a,b,\Nn}^{\M,U_F}$ and
\begin{equation}\label{pf1}
\sum _{n=0}^{N-n_F} \frac{q_n(\lambda^{a,b}(s))q_n(\lambda^{a,b}(\tilde s))}{\Vert q_n\Vert _2 ^2}=\langle \phi_s,\phi_{\tilde s}\rangle _{\nu_{a,b,\Nn}^{\M, U_F}}=\frac{1}{\nu_{a,b,\Nn}^{\M,U_F}(s)}\delta_{s,\tilde s}.
\end{equation}
This is the dual orthogonality associated to the orthogonality
$$
\sum_{s\in \supp_{\nu_{a,b,\Nn}^{\M,U_F}}}q_n (\lambda^{a,b}(s))q_m (\lambda^{a,b}(s))\nu_{a,b,\Nn}^{\M, U_F}(s)=\langle q_n ,q_n \rangle \delta_{n,m}
$$
of the polynomials $q_n $, $n=0,\cdots, N-n_F$, with respect to the positive measure $\nu_{a,b,\Nn}^{\M,U_F}$ (see, for instance, \cite{At}, Appendix III, or \cite{KLS}, Th. 3.8).

Using the duality (\ref{duaqnrn}) in Lemma \ref{lem3.2} and  (\ref{norkdh}), we get from (\ref{pf1}) (after careful computations)
\begin{align*}
\sum _{n=0}^{N-n_F}h_{v+u_F}^{\ax,\bx, N;\M,F}&(n)h_{\tilde v+u_F}^{\ax,\bx, N;\M,F}(n) \omega_{\alpha,\beta, N}^{F}(n)\\\nonumber
&\quad =\frac{\tau_v\zeta_{v+\ax+\bx}^2(\theta_v^\M)^2\prod_{j=1}^{-\bx}(N+\ax+\bx+j)^2}
{n_{U_F}!(-\ax+\bx+n_{U_F})!\nu_{a,b,\Nn}^{\M}(v+\ax+\bx)}\delta_{v,\tilde v}.
\end{align*}
This shows that the  polynomials $h_v^{\ax,\bx, N;\M,F}$, $v\in \sigma_{N;F}$, are orthogonal and have non-null  $L^2$ norm (actually, the admissibility conditions in Definition \ref{laad} show that the norm is positive). Since the positive measure $\omega_{\ax,\bx, N}^{\M,F}$  has $N-n_F+1$ points in its support and we have $N-n_F+1$ polynomials $h_v^{\ax,\bx, N;\M,F}$ of degree $v$, we can conclude using part 2 of Lemma \ref{ldmp} that they form an orthogonal basis in $L^2(\omega_{\ax,\bx, N}^{\M,F})$.

\end{proof}

\section{New exceptional Jacobi families depending on an arbitrary number of continuous parameters}\label{sec5}
As in the rest of this paper, $\M$ denotes the set of parameters $\M=\{M_0,M_1,\cdots \}$, $F$ a finite set of positive integers, and $\ax$ and $\bx$ denote negative integers.

\begin{definition}\label{jll}
We associate to $\ax,\bx$, $\M$ and $F$ the sequence of polynomials
\begin{equation}\label{losjx2}
P_n^{\ax, \bx;\M, F}(x)=\left|
  \begin{array}{@{}c@{}lccc@{}c@{}}
    & (\Pp_{n-u_F}^{\ax ,\bx;\M})^{(j-1)}(x) &&\hspace{-.4cm}{}_{1\le j\le n_F+1} \\
    \dosfilas{(\Pp_{f}^{\ax ,\bx ;\M})^{(j-1)}(x) }{f\in F} \\
  \end{array}
  \hspace{-.4cm}\right|
\end{equation}
where $n\in \sigma _F$ (\ref{elsig}) and $(\Pp_n^{\ax,\bx;\M})_n$ are the polynomials introduced in Definitions \ref{losj} (when $\ax\le \bx$) and \ref{losjt} (when $\bx\le \ax$).
\end{definition}

We next prove that the polynomial $P_n^{\ax, \bx;\M, F}(x)$ can be obtained (up to normalization constants) from the polynomial $h_n^{\ax,\bx,N;\M,F}(x)$ (\ref{defmexi2}) setting $x\to (1-x)N/2$ and taking limit as $N\to +\infty$.

\begin{lemma}\label{tol}
For $n\in \sigma_F$,
\begin{equation}\label{lim1t}
\lim_{N\to +\infty}\frac{h_n^{\ax,\bx,N;\M,F}\left((1-x)N/2\right)}{N^n}=\upsilon_n^{F}P_n^{\ax,\bx ;\M,F}(x)
\end{equation}
uniformly in compact sets, where
\begin{equation}\label{defupn}
\upsilon_n^{F}=(-1)^{n}2^{\binom{n_F+1}{2}}(n-u_F)!\prod_{f\in F}f!.
\end{equation}
\end{lemma}

\begin{proof}
The Lemma is an easy consequence of the following result.

\noindent
For a finite set $K$ of positive integers define
$$
P_K(x)=\left|
  \begin{array}{@{}c@{}lccc@{}c@{}}
    &  &&\hspace{-1.4cm}{}_{1\le j\le n_K} \\
    \dosfilas{(\Pp_{k}^{\ax ,\bx ;\M})^{(j-1)}(x) }{k\in K} \\
  \end{array}
  \hspace{-.4cm}\right|,\quad  h_K(x)= \left|\begin{array}{@{}c@{}lccc@{}c@{}}
    & &&\hspace{-1.4cm}{}_{1\le j\le n_K} \\
    \dosfilas{\Hh_{k}^{\ax ,\bx ,N;\M}(x+j-1) }{k\in K} \\
  \end{array}  \hspace{-.4cm}\right|,
$$
and set $g_K=u_K+n_K$. Then
\begin{equation}\label{ldm}
\lim_{N\to \infty}\frac{h_K((1-x)N/2)}{N^{g_K}}=\upsilon_K P_K(x),\quad x\in \CC,
\end{equation}
where
\begin{equation}\label{upn0}
\upsilon_K=(-1)^{g_K}2^{\binom{n_{K}}{2}}\prod_{k\in K}k!.
\end{equation}
We prove (\ref{ldm}) in two steps.
\medskip

\noindent
\textit{Step 1.} Let $p_N$, $N\in \NN$, and $q$ be two polynomials of degree $g$ satisfying that
\begin{equation}\label{pldm}
\lim_{N\to \infty}\frac{p_N((1-x)N/2)}{N^g}=q(x),
\end{equation}
uniformly in compact set of $\CC$.
Write $r_N(x)=p_N(x+1)-p_N(x)$. Then
\begin{equation}\label{sldm}
\lim_{N\to \infty}\frac{r_N((1-x)N/2)}{N^{g-1}}=-2q'(x),
\end{equation}
uniformly in compact set of $\CC$.

Indeed, if we write $p_N(x)=\sum_{j=0}^ga_j^Nx^j$ and $q(x)=\sum_{j=0}^gb_j(1-x)^j$, we easily get from (\ref{pldm}) that
\begin{equation}\label{tldm}
\lim_{N\to \infty}\frac{a_j^N}{2^jN^{g-j}}=b_j,\quad j=0,\cdots, g.
\end{equation}
We now write
\begin{align*}
r_N((1-x)N/2)&=\sum_{j=0}^ga_j^N\left(\frac{(1-x)N}{2}+1\right)^j-\sum_{j=0}^ga_j^N\left(\frac{(1-x)N}{2}\right)^j\\
&=\sum_{j=0}^ga_j^N\sum_{l=0}^{j-1}\binom{j}{l}\left(\frac{(1-x)N}{2}\right)^l\\
&=\sum_{l=0}^{g-1}\left(\frac{(1-x)N}{2}\right)^l\sum_{j=l+1}^{g-1}\binom{j}{l}a_j^N.
\end{align*}
Using (\ref{tldm}), we deduce that
$$
\lim_{N\to \infty}\frac{\sum_{j=l+1}^{g-1}\binom{j}{l}a_j^N}{2^lN^{g-l-1}}=2(l+1)b_{l+1},\quad l=0,\cdots, g-1,
$$
from where (\ref{sldm}) follows easily.

\medskip

\noindent
\textit{Step 2.} We proceed by induction on $n_K$. For $n_K=1$, (\ref{ldm}) is just (\ref{blmel2}).

Assume now that (\ref{ldm}) holds for any finite set $K$ of positive integers with less than $s$ elements.

If $n_K=s$, write
$$
K_0=K\setminus\{\min K\}, K_1=K\setminus\{\max K\}, K_{0,1}=K\setminus\{\min K,\max K\}.
$$
Notice that $g_K=u_K+n_K$ is the degree of $h_K$ (see (\ref{lcrnw})). An easy computation using (\ref{elsig}) shows that
\begin{equation}\label{lgc}
g_K=g_{K_1}+g_{K_0}-g_{K{0,1}}-1.
\end{equation}
Applying Sylvester's identity in Lemma \ref{lemS} to $h_K$ (for $i_0=1,i_1=n_K$ and $j_0=1,j_1=n_K$) and using (\ref{lgc}), we get
\begin{align*}
\frac{h_K(x)}{N^{g_K}}&=\frac{h_{K_1}(x+1)h_{K_0}(x)-h_{K_1}(x)h_{K_0}(x+1)}{N^{g_K}h_{K_{0,1}}(x+1)}\\
&=\frac{\frac{h_{K_1}(x+1)-h_{K_1}(x)}{N^{g_{K_1}-1}}\frac{h_{K_0}(x)}{N^{g_{K_0}}}-\frac{h_{K_0}(x+1)-h_{K_0}(x)}{N^{g_{K_0}-1}}\frac{h_{K_1}(x)}{N^{g_{K_1}}}}
{\frac{h_{K_{0,1}}(x+1)}{N^{g_{K_{0,1}}}}}.
\end{align*}
Setting $x\to (1-x)N/2$, taking limit as $N\to +\infty$ and using the induction hypothesis and the first step, we deduce
$$
\lim_{N\to \infty}\frac{h_K((1-x)N/2)}{N^{g_K}}=\frac{-2\upsilon_{K_0}\upsilon_{K_1}}{\upsilon_{K_{0,1}}}\frac{P_{K_1}'(x)P_{K_0}(x)-P_{K_1}(x)P_{K_0}'(x)}{P_{K_{0,1}}(x)}.
$$
Applying again Sylvester's identity in Lemma \ref{lemS} to $P_K$ (for $i_0=1,i_1=n_K$ and $j_0=n_K-1,j_1=n_K$) and using (\ref{upn0}),
we finally deduce that the expression in the right hand side of the previously identity is
$$
\upsilon_KP_K(x).
$$

\end{proof}

As a consequence of the Lemma, we deduce that $P_n^{\ax,\bx ;\M,F}$ is a polynomial of degree $n$ with leading coefficient equal to
\begin{equation*}\label{lcrnp}
V_{F}\prod_{i\in \{n-u_F\},F_1}s_i^{\ax,\bx;\M}\prod_{f\in F}(f-n+u_F),
\end{equation*}
where $V_F$ is the Vandermonde determinant (\ref{defvdm}) and $s_i^{\ax,\bx;\M}$  is the leading coefficient of the  polynomial $\Pp_i^{\ax,\bx,N;\M}$ (see (\ref{lcja})).

As for the exceptional Hahn polynomials, we only have to consider the case $\ax\le \bx$ because it follows easily from (\ref{fi2ja}) that
$$
P_n^{\ax, \bx;\M, F}(x)=(-1)^nP_n^{\bx, \ax;\M^{-1}, F}(-x).
$$
Hence, from now on, we also assume $\ax\le\bx$.

There are some reasons (which we explain in Section \ref{secu}) to assume also that
\begin{equation}\label{cis3}
\{-\bx,\cdots, -\ax-\bx-1\}\subset F.
\end{equation}

According to Definition \ref{losj}, the polynomials $P_n^{\ax,\bx;\M, F}$, $n\in \sigma _F$, seem to depend on the parameters $M_{i}$,  $i=0,1,\cdots,-\bx-1$. However, as for the exceptional Hahn family, the polynomials only depend on the parameters
$M_i$, $i\in F_{\bx}$ (\ref{conp}). Indeed, if $f\in F$, $-\ax\le f\le -\ax-\bx-1$ and $-\ax-\bx-f-1\in F$ then we can use the polynomial $\Pp_{-\ax-\bx-f-1}^{\ax,\bx;\M}$ in the determinant (\ref{losjx2}) to remove the second summand in the right hand side of the identity (\ref{losj2}) which defines the polynomial $\Pp_{f}^{\ax,\bx;\M}$. In doing that we remove the dependence of the polynomial $P_{n}^{\ax,\bx;\M,F}$ on the parameter $M_{\ax+f}$.
As for the exceptional Hahn polynomials, we can be more precise: enumerate the polynomials $P_n^{\ax,\bx;\M, F}$, $n\in \sigma _F$, in accordance to the position of $n$ in the set $\sigma _F$ (i.e., the first polynomial would be $P_{u_F}^{\ax,\bx;\M, F}$)
and similarly enumerate the parameters $M_i$, $i\in F_{\bx}$, in accordance to the position of $i$ in the set $F_b$ (\ref{conp2}). it is then not difficult to check that for $i=1,\cdots, n_{F_b}$, the $i$-th polynomial $P_{n}^{\ax,\bx;\M,F}$ does not depend on the $(n_{F_b}-i)$-th parameter, and for $i\ge n_{F_b}+1$, the $i$-th  polynomial $P_{n}^{\ax,\bx;\M,F}$  depends on all the parameters $M_i$, $i\in F_b$.

We introduce the associated polynomial
\begin{equation}\label{defhom}
\Omega _{\M,F}^{\ax,\bx}(x)=\left|
  \begin{array}{@{}c@{}lccc@{}c@{}}
    &  &&\hspace{-.9cm}{}_{1\le j\le n_F} \\
    \dosfilas{ (\Pp_{f}^{\ax ,\bx })^{(j-1)}(x) }{f\in F}
  \end{array}
  \hspace{-.4cm}\right|.
\end{equation}
As a consequence of Lemma \ref{tol}, we have
\begin{align}\label{lim2}
\lim _{N\to +\infty}\frac{\Omega _{\M,F}^{\ax,\bx,N}((1-x)N/2)}{N^{u_F+n_F}}&=\upsilon_{F}\Omega _{\M,F}^{\ax,\bx} (x),
\end{align}
uniformly in compact set of $\CC$, where $\upsilon_F$ is defined by (\ref{upn0}).

Notice that $\Omega_{\M,F}^{\ax,\bx}$ is a polynomials of degree $u_F+n_F$ and leading coefficient
\begin{equation*}\label{lcornp}
V_{F}\prod_{i\in F_1}s_i^{\ax,\bx;\M},
\end{equation*}
where $V_F$ is the Vandermonde determinant (\ref{defvdm}) and $s_i^{\ax,\bx;\M}$  is the leading coefficient of the  polynomial $\Pp_i^{\ax,\bx,N;\M}$ (see (\ref{lcja})).

The following property will be useful to show that the exceptional Legendre polynomials introduced in \cite{xle} are particular cases of the exceptional Jacobi polynomials introduced here.

\begin{remark}\label{jlp1} We first renormalize the polynomials $P_{n}^{\ax,\bx;\M,F}$ as follows:
$$
\bar P_{n}^{\ax,\bx;\M,F}(x)=(\prod_{i=0}^{-\bx -1}(M_i-1))P_{n}^{\ax,\bx;\M,F}(x),\quad n\in \sigma_F.
$$
As in Remark \ref{hlp1}, for a finite set $J$ of nonnegative integers, we denote by $\M_J$ the particular case of the set of parameters $\M$ obtained by setting $M_j=1$, $j\in J$.

If $f\in\{-\ax,\cdots,-\ax-\bx-1\}\cap F$ and $-f-\ax-\bx-1\not \in F$, write $\tilde F=(F\setminus\{f\})\cup {-f-\ax-\bx-1}$. Then
for $n\in \sigma_F$, $n\not =u_F-f-\ax-\bx-1$
$$
\bar P_{n-(2f+\ax+\bx+1)}^{\ax, \bx;\M, \tilde F}(x)=\frac{(-1)^{n_f}(M(\ax+f)-1)}{d_f}\bar P_n^{\ax, \bx;\M_{\{\ax+f\}}, F}(x),
$$
where
$$
d_f=\frac{(f+\ax)!(f+\bx)!(-\ax-\bx-f-1)!}{(-1)^{\bx+f}f!},
$$
and $n_f$ denotes the number of elements in $F$ which are bigger than  $-f-\ax-\bx-1$ and less than $f$;
similarly
$$
P_{u_F-f-\ax-\bx-1}^{\ax, \bx;\M, \tilde F}(x)=(-1)^{n_f-1}P_{u_F-f-\ax-\bx-1}^{\ax, \bx;\M_{\{\ax+f\}}, F}(x).
$$
The proof is analogous to that of Remark \ref{hlp1}.
\end{remark}

\subsection{The second order differential operator}
Passing to the limit, we can transform the second order difference operator (\ref{sodomex}) for the polynomials $h_n^{\ax,\bx,N;\M,F}$, $n\in\sigma_F$, in a second order differential operator with respect to which the polynomials $P_n^{\ax,\bx ;\M, F}$, $n\in\sigma_F$, are eigenfunctions.

\begin{theorem}\label{th5.1} The polynomials $P_n^{\ax,\bx;\M, F}$, $n\in \sigma _F$,
are common eigenfunctions of the second order differential operator
\begin{align*}
D&=(1-x^2)\frac{d^2}{dx^2}+h_1(x)\frac{d}{dx}+h_0(x),\\\nonumber
h_1(x)&=\bx-\ax -(\ax+\bx+2n_F+2)x-2(1-x^2)\frac{(\Omega_{\M,F}^{\ax,\bx})'(x)}{\Omega_{\M,F}^{\ax,\bx}(x)},\\\nonumber
h_0(x)&=-\lambda^{\ax,\bx}(n_F) +[\ax-\bx+(2n_F+\ax +\bx)x]\frac{(\Omega_{\M,F}^{\ax,\bx})'(x)}{\Omega_{\M,F}^{\ax,\bx}(x)}+(1-x^2)\frac{(\Omega_{\M,F}^{\ax,\bx})''(x)}{\Omega_{\M,F}^{\ax,\bx}(x)}.
\end{align*}
More precisely $D(P_n^{\ax,\bx;\M, F})=-\lambda^{\ax,\bx}(n- u_F)P_n^{\ax,\bx;\M, F}(x)$.
\end{theorem}

\begin{proof}
We omit the proof because proceeds as that of Theorem 5.1 in \cite{duch}, using the basic limit (\ref{blmel}).

\end{proof}

\subsection{Orthogonality of the polynomials $P_n^{\ax,\bx ;\M, F}$, $n\in\sigma_F$}
The key concept for the existence of a positive measure with respect to which the polynomials $(P_n^{\ax,\bx;\M,F})_n$ are orthogonal is that of admissibility (see Definition \ref{laad}). Admissibility is implied by the fact that  the polynomial $\Omega _{\M,F} ^{\ax,\bx}(x)$ does not vanish in the interval $[-1,1]$.

\begin{lemma}\label{pdf} Let $\ax,\bx$ be nonnegative integers with $\ax\le \bx$ and let $F$ be a finite set of positive integers satisfying (\ref{cis}).
If $\Omega_{\M,F}^{\ax,\bx}(x)\not =0$, $x\in[-1,1]$, then  $\ax$, $\bx$, $\M$ and $F$ are admissible.
\end{lemma}

\begin{proof}
Consider the polynomial $\Omega_{\M,F}^{\ax,\bx,N}$ (\ref{defom}), and write $s=u_F+n_F$ for its degree.
As a consequence of (\ref{lim2}), we deduce that the sequence of derivatives
$$
\left(\frac{(\Omega_{\M,F}^{\ax,\bx,N})'((1-x)N/2)}{N^{s-1}}\right)_N
$$
is uniformly bounded when $x\in[-1,1]$. Since when $x$ runs in $[-1,1]$, the number $(1-x)N/2$ runs in $[0,N]$, we can find a positive constant $C>0$ such that
\begin{equation}\label{ppat}
\sup _{y\in[0,N]}\left\vert \frac{(\Omega_{\M,F}^{\ax,\bx,N})'(y)}{N^{s-1}} \right\vert\le C.
\end{equation}
We now proceed by \textit{reductio to absurdum}. Assume then that $\ax$, $\bx$, $\M$ and $F$ are not admissible. According to the part 3 of Lemma \ref{l3.1}, for $N\in \NN$ big enough, there exists $n$, $0\le n\le N-n_F$, such that $\Omega_{\M,F}^{\ax,\bx,N}(n)\Omega_{\M,F}^{\ax,\bx,N}(n+1)<0$. Hence there exists $z_N$, $0\le z_N\le N$ such that $\Omega_{\M,F}^{\ax,\bx,N}(z_N)=0$. Since $0\le z_N/N\le 1$, we can find a sequence $N_k$ such that $N_k\to +\infty$ and
$z_{N_k}/N_k\to \zeta\in[0,1]$ as $k\to +\infty$. Write $z_{N_k}=(1-x_{N_k})N_k/2$, so that $x_{N_k}=1-2z_{N_k}/N_k\to \theta=1-2\zeta\in[-1,1]$ as $k\to +\infty$. Hence, by applying the mean value theorem, we have
\begin{align*}
\vert \Omega_{\M,F}^{\ax,\bx,N_k}((1-\theta)N_k/2)\vert&=\vert \Omega_{\M,F}^{\ax,\bx,N_k}((1-x_{N_k})N_k/2)-\Omega_{\M,F}^{\ax,\bx,N_k}((1-\theta)N_k/2)\vert\\
&=\frac{N_k}{2}\vert (\Omega_{\M,F}^{\ax,\bx,N_k})'(y)(x_{N_k}-\theta)\vert,
\end{align*}
for certain $y\in(z_{N_k},(1-\theta)N_k/2)\subset [0,N_k]$. Using (\ref{ppat}), we get
$$
\frac{\vert \Omega_{\M,F}^{\ax,\bx,N_k}((1-\theta)N_k/2)\vert}{N^s}\le C\frac{\vert x_{N_k}-\theta\vert}{2}.
$$
Taking limit when $k$ goes to $+\infty$, we deduce using (\ref{lim2}) that $\Omega_{\M,F}^{\ax,\bx}(\theta)=0$ with $\theta\in[-1,1]$,
which it is a contradiction

\end{proof}

We guess that the converse of Lemma \ref{pdf} is also true and proposed it as a conjecture  in the Introduction. We have a lot of
computational support for it but no proof yet.

The fact that
\begin{equation}\label{ona}
\Omega_{\M,F}^{\ax,\bx}(x)\not =0,\quad x\in[-1,1],
\end{equation}
implies the existence of a positive weight which respect to which the  polynomials $P_n^{\ax,\bx;\M,F}$, $n\in \sigma_F$, are orthogonal and complete with respect to a positive weight in $[-1,1]$.

\begin{theorem}\label{th5.6} Let $\ax$ and $\bx$ be two negative integers and a positive integer
satisfying $\ax\le \bx\le -1$, and let $F$ be a finite set of positive integers such that (\ref{cis}) holds.
If we also assume (\ref{ona}), then
the  polynomials $P_n^{\ax,\bx;\M,F}$, $n\in \sigma_F$, are orthogonal and complete with respect to the positive weight
$$
\omega_{\ax,\bx;\M, F} =\frac{(1-x)^{\ax +n_F}(1+x)^{\bx+n_F}}{(\Omega_{\M, F}^{\ax,\bx}(x))^2},\quad x\in[-1,1].
$$
Moreover for $n\not \in F$
$$
\Vert P_{n+u_F}^{\ax,\bx;\M,F} \Vert_2=\frac{2^{\ax+\bx+1}\varrho^\M_n(\prod_{f\in F}(n-f))(\prod_{u\in U_F}(n+u+1))\prod_{i=1}^{-\ax}(n+\ax+\bx+i)}{2n+\ax+\bx+1},
$$
where
$$
\varrho^\M_n=\begin{cases}-\frac{(1-M_{-\bx-1-n})^2}{M_{-\bx-1-n}},& \mbox{if $0\le n\le -\bx-1$},\\ 1,&\mbox{otherwise}.\end{cases}
$$
\end{theorem}

\begin{proof}
Note that the assumption (\ref{cis}) implies that $\ax+n_F,\bx+n_F\ge 0$, and then
the positive weight $\omega_{\ax,\bx;\M, F}\in L^1([-1,1])$.

Since the proof is similar to other cases of exceptional polynomials, we only sketch it.

We first prove the identity for the $L^2$-norm of the polynomial $P_{n+u_F}^{\ax,\bx;\M,F}$, $n\not \in F$.

Fixed a nonnegative integer $n\not \in F$, i.e. $n+u_F\in \sigma_F$, since $\ax$, $\bx$, $\M$ and $F$ are admissible, we consider the positive measure defined by
\begin{equation}\label{defmt}
\tau _{N}=\sum _{x=0}^{N-n_F}\frac{\binom{\ax+n_F+x}{x}\binom{\bx+N-x}{N-n_F-x}(h_{n+u_F}^{\ax,\bx,N;\M,F}(x))^2}{\Omega _{\M,F}^{\ax,\bx,N}(x)\Omega_{\M,F}^{\ax,\bx,N}(x+1)}\delta {y_{N,x}},
\end{equation}
where
\begin{equation}\label{defya}
y_{N,x}=1-2x/N.
\end{equation}

We need the following limits
\begin{align}\label{lm1}
\lim _{N\to +\infty}\frac{\Omega _{\M,F}^{\ax,\bx,N}((1-x)N/2)}{N^{u_F+n_F}}&=\upsilon_{F}\Omega _{\M,F}^{\ax,\bx} (x),\\\label{lm11}
\lim _{N\to +\infty}\frac{\Omega _{\M,F}^{\ax,\bx,N}((1-x)N/2+1)}{N^{u_F+n_F}}&=\upsilon_{F}\Omega _{\M,F}^{\ax,\bx} (x),\\\label{lm2}
\lim _{N\to +\infty}\frac{h _{n+u_F}^{\ax,\bx,N;\M,F}((1-x)N/2)}{N^{n+u_F}}&=\upsilon^{F}_{n+u_F}P_{n+u_F}^{\ax,\bx;\M,F}(x),\\\label{lm3}
\lim _{N\to +\infty}\frac{\binom{\ax+n_F+(1-x)N/2}{(1-x)N/2}\binom{\bx+N-(1-x)N/2}{N-n_F-(1-x)N/2}}{N^{\ax+\bx+2n_F}}&=\frac{(1-x)^{\ax+n_F}
(1+x)^{\bx+n_F}}{c},
\end{align}
uniformly in the interval $[-1,1]$, where $\upsilon^{F}_{n}$ and $\upsilon_{F}$ are defined by (\ref{defupn}) and (\ref{upn0}), respectively, and $c$ is given by
\begin{equation}\label{defc}
c=2^{\ax+\bx+2n_F}
(\ax+n_F)!(\bx+n_F)!.
\end{equation}
The  first limit is (\ref{lim2}). The second one is a consequence of the firs step in the proof of Lemma \ref{tol}. The third one is (\ref{lim1t}). The forth one is consequence of the asymptotic behavior of $\Gamma(z+u)/\Gamma(z+v)$ when $z\to\infty$ (see \cite{EMOT}, vol. I (4), p. 47).

Since $\Omega _{\M,F}^{\ax,\bx} $ does not vanish in $[-1,1]$, applying Hurwitz's Theorem to the limits (\ref{lm1}) and (\ref{lm11}) we can choose a sequence $N_k$ of positive integers such that $N_k\to +\infty$ as $k\to \infty$ and  $\Omega _{\M,F}^{\ax,\bx,N_k}((1-x)N_k/2)\Omega _{\M,F}^{\ax,\bx,N_k}((1-x)N_k/2+1)\not =0$, $x\in [-1,1]$.

Hence, using (\ref{defupn}) and (\ref{upn0}), we can combine the limits (\ref{lm1}), (\ref{lm11}), (\ref{lm2}) and (\ref{lm3}) to get
\begin{equation}\label{lm4}
\lim _{k\to +\infty}H_{N_k}(x)=\frac{4^{n_F}}{c}H(x),\quad \mbox{uniformly in $[-1,1]$, where}
\end{equation}

\begin{align*}
H_{N_k}(x)&=\frac{\binom{\ax+n_F+(1-x)N_k/2}{(1-x)N_k/2}\binom{\bx+N_k-(1-x)N_k/2}{N_k-n_F-(1-x)N_k/2}(h_{n+u_F}^{\ax,\bx,N_k;\M,F}((1-x)N_k/2))^2}{N^{\ax+\bx+2n}\Omega _{\M,F}^{\ax,\bx,N_k}((1-x)N_k/2)\Omega_{\M,F}^{\ax,\bx,N_k}((1-x)N_k/2+1)},\\
H(x)&=\frac{(1-x)^{\ax+n_F}(1+x)^{\bx+n_F}(P_{n+u_F}^{\ax,\bx;\M,F})^2(x)}
{(\Omega_{\M,F}^{\ax,\bx})^2(x)}.
\end{align*}

We now prove that
\begin{equation}\label{lm5}
\lim _{k\to +\infty}\frac{2\tau _{N_k}([-1,1])}{N_k^{\ax+\bx+n+1}}=\frac{4^{n_F}}{c}\int_{-1}^1H(x)dx.
\end{equation}
To do that, write $I_{N_k}=\{ l\in \NN: 0\le l\le N_k\}$, ordered in decreasing size.
The numbers $y_{N_k,l}$, $l\in I_{N_k}$, form a partition of the interval $[-1,1]$ with $y_{N_k,l+1}-y_{N_k,l}=2/N_k$ (see (\ref{defya})). Since the function $H$ is continuous  in the interval $[-1,1]$, we get that
$$
\int_{-1}^1H(x)dx=\lim_{k\to +\infty}S_{N_k},
$$
where $S_{N_k}$ is the Cauchy sum
$$
S_{N_k}=\sum_{l\in I_{N_k}}H(y_{N_k,l})(y_{N_k,l+1}-y_{N_k,l}).
$$
On the other hand, since $l\in I_{N_k}$ if and only if $-1\le y_{N_k,l}\le 1$ (\ref{defya}), we get
\begin{align*}
\frac{2\tau _{N_k}([-1,1])}{N_k^{\ax+\bx+n+1}}&=\frac{2}{N_k^{\ax+\bx+n+1}}\sum _{l\in I_{N_k}}\frac{\binom{\ax+n_F+l}{l}\binom{\bx+N_k-l}{N_k-n_F-l}(h_{u_F}^{\ax,\bx,N_k;\M,F})^2(l)}{\Omega _{\M,F}^{\ax,\bx,N_k}(l)\Omega _{\M,F}^{\ax,\bx,N_k}(l)}\\
&=\frac{2}{N_k}\sum _{l\in I_{N_k}}H_{N_k}(y_{N_k,l})=\sum _{l\in I_{N_k}}H_{N_k}(y_{N_k,l})(y_{N_k,l+1}-y_{N_k,l}).
\end{align*}
The limit (\ref{lm5}) now follows from the uniform limit (\ref{lm4}).

The formula for the $L^2$-norm of the polynomial $P_{n+u_F}^{\ax,\bx;\M,F}$ follows now by using
the formula for the $L^2$-norm of the polynomial $h_{n+u_F}^{\ax,\bx,N;\M,F}$ (\ref{nomh}), the definition of the measure $\nu_{a,b,\Nn}^{\M, F}$ (\ref{lanu}), and some careful computations.

\medskip

Write $\Aa $ for the linear space generated by the polynomials $P_n^{\ax,\bx ;\M,F}$, $n\in \sigma _F$.
One can check, using Lemma 2.6 of \cite{duch}, that the second order differential operator $D$ in Theorem \ref{th5.1} is symmetric with respect to the pair $(\omega _{\ax,\bx;\M,F}, \Aa)$ (in the sense that
$$
\langle D(p),q\rangle_{\omega _{\ax,\bx;\M,F}}=\langle p,D(q)\rangle_{\omega _{\ax,\bx;\M,F}},\quad p,q\in \Aa).
$$
Since the polynomials $P_n^{\ax,\bx;\M,F}$, $n\in \sigma _F$, are eigenfunctions of $D$ with different eigenvalues, Lemma 2.4 of \cite{duch} implies that they are orthogonal with respect to $\omega _{\ax,\bx ;\M,F}$.

In order to prove the completeness of $P_n^{\ax,\bx,;\M,F}$, $n\in \sigma _F$, we proceed in two steps.

\bigskip
\noindent
\textsl{Step 1.} Write $\Bb=\{(\Omega _{\M,F}^{\ax,\bx}(x))^2p: p\in \PP\}$. Then $\Bb$ is dense in $L^2(\omega _{\ax,\bx;\M,F})$.

Take a function $f\in L^2(\omega _{\ax,\bx;\M,F})$ and  define the function
$$
g(x)=f(x)/(\Omega _{\M,F}^{\ax,\bx})^2\in L^2((1-x)^{\ax+n_F}(1+x)^{\bx+n_F}).
$$
Given $\epsilon >0$, since the polynomials are dense in $L^2((1-x)^{\ax+n_F}(1+x)^{\bx+n_F})$, there exists a polynomial $p$ such that
\begin{equation}\label{dmp}
\int_{-1}^1 \vert g(x)-p(x)\vert ^2(1-x)^{\ax+n_F}(1+x)^{\bx+n_F}dx<\epsilon.
\end{equation}
We then have
\begin{align*}
\int_{-1}^1& \vert g(x)-p(x)\vert ^2(1-x)^{\ax+n_F}(1+x)^{\bx+n_F}dx\\&=\int_{-1}^1  \left\vert f(x)-(\Omega _{\M,F}^{\ax,\bx}(x))^2p(x)\right\vert ^2\omega _{\ax,\bx;\M,F}dx .
\end{align*}
Using (\ref{dmp}), we can conclude that $\Bb$ is dense in $L^2(\omega _{\ax,\bx;\M,F})$.

\bigskip
\noindent
\textsl{Step 2.} $\Aa\subset \Bb$.

This step can be proved as Lemma 1.1 in \cite{duco}.

The Theorem follows now from the first step.

\end{proof}

We finish this Section comparing the exceptional Jacobi polynomials constructed in this paper with the exceptional Legendre polynomials introduced in \cite{xle}.

In \cite{xle} the authors construct the exceptional Legendre families by the application of a finite number of confluent Darboux transformations to the Legendre second order differential operator. The approach is hence completely different to the one used here. For a $n$-tuple $\mathbf{m}=(m_1,\cdots, m_n)$, they associated $n$ real parameters $\mathbf{t_m}=\{t_1,\cdots ,t_n\}$ (which play the role of our set of parameters $\M$) and define a sequence $P_{\mathbf{m};i}(z;\mathbf{t_m})$, $i\in \NN\setminus J$, of polynomials, where $J$ is certain finite set of nonnegative integers. Their definition of the exceptional Legendre polynomial $P_{\mathbf{m};i}(z;\mathbf{t_m})$ neither use Wronskian nor
the polynomials $\Pp_n^{\ax,\bx;\M,F}$ introduced in Definition \ref{losj}.

Instead of that, they proceed as follows.
Define the $n\times n$ matrix polynomial
$$
\mathcal{R}_{\mathbf{m}}(z;\mathbf{t_m})=\left(
  \begin{array}{@{}c@{}lccc@{}c@{}}
    & &&\hspace{-.4cm}{}_{1\le l\le n} \\
    \dosfilas{\delta_{k,l}+t_{m_l}R_{m_k,m_l}(z) }{1\le k\le n} \\
      \end{array}
  \hspace{-.3cm}\right),
$$
where
$$
R_{m_k,m_l}(z)=\int _{-1}^zP_{m_k}(u)P_{m_l}(u)du,
$$
and $P_i$ denote the classical Legendre polynomials. They denote by $\tau _{\mathbf{m}}(z;\mathbf{t_m})$ the determinant of
$\mathcal{R}_{\mathbf{m}}(z;\mathbf{t_m})$ and define the $n$-tuple of polynomials
$$
\mathbf{Q_m}^T(z;\mathbf{t_m})=\tau _{\mathbf{m}}(z;\mathbf{t_m})\mathcal{R}_{\mathbf{m}}(z;\mathbf{t_m})^{-1}(P_{m_1}(z),\cdots,P_{m_n}(z))^T,
$$
and finally
$$
P_{\mathbf{m};i}(z;\mathbf{t_m})=[\mathbf{Q}_{(\mathbf{m},i)}(z;\mathbf{t}_{(\mathbf{m},i)})]_{n+1},
$$
where $(\mathbf{m},i)=(m_1,\cdots, m_n,i)$ and $\mathbf{t}_{(\mathbf{m},i)}=\{t_1,\cdots ,t_n,t_i\}$
(see \cite[Definition 3]{xle}).

They prove that these polynomials $(P_{\mathbf{m};i}(z;\mathbf{t_m}))_i$ are exceptional Legendre polynomials, that is, they are eigenfunctions of a second order differential operator, and under mild conditions on the parameters $\mathbf{t_m}$, they are also orthogonal in $[-1,1]$ with respect to the positive weight
$$
\frac{1}{\tau _{\mathbf{m}}(z;\mathbf{t_m})^2}.
$$
As one can see this definition is completely different to our Definition \ref{jll}. In fact, we have not been able to prove that the exceptional polynomials constructed in \cite{xle} are particular cases of our exceptional Jacobi polynomials, although we have plenty of computation evidence showing that this is the case.

For instance, using Maple, we have been able to check that for a positive integer $m_1$, the one parametric exceptional Legendre family in \cite[Section 4.1]{xle}
correspond to our polynomials $P_n^{\ax,\bx;\M,F}$, where
\begin{equation}\label{gpe}
\ax=\bx=-m_1-1,\quad F=\{1,\cdots, m_1,2m_1+1\},\quad M_{m_1}=\frac{2m_1+1}{2m_1+1+2t_{m_1}}.
\end{equation}
Notice that the finite set $F$ of positive integers in (\ref{gpe}) does not satisfy (\ref{cis}), but using the Remark \ref{jlp1}, one can check that the exceptional Legendre polynomials associated to $m_1$ are actually  the particular case of
$$
\ax=\bx=-m_1-1,\quad F=\{m_1+1,\cdots,2m_1+1\},
$$
when $M_i=1$, $i=0,\cdots , m_1-1,$ and $M_{m_1}=\frac{2m_1+1}{2m_1+1+2t_{m_1}}$.

\section{The assumption $\{-\bx,\cdots,-\ax-\bx-1\}\subset F$}\label{secu}
When the assumption $\{-\bx,\cdots,-\ax-\bx-1\}\subset F$ (\ref{cis}) on the finite set of positive integers $F$ does not hold, one can still associated to $F$ the sequence of polynomials $h_n^{\ax,\bx,N;\M,F}$ or $P_n^{\ax,\bx;\M,F}$ as in Definitions \ref{dll} or \ref{jll}, respectively. Although we have proved that they are eigenfunctions of a second order difference or differential operator only when $F$ satisfies (\ref{cis}), the result seems to be always true. In fact, using \cite[Remar 5.3]{dundh}, the duality in Lemma \ref{lem3.2} can be extended for many other sets $F$ which do not satisfy (\ref{cis}), and then Theorems \ref{th3.3} and \ref{th5.1} are also true for these sets $F$.

However, there are a number of reasons showing that the case when $F$ does not satisfy (\ref{cis}) is not very much interesting.

\begin{enumerate}
\item The assumption (\ref{cis}) implies that $0\le \ax+n_F\le \bx+n_F$, and this is a necessary condition for defining the measures with respect to which our exceptional Hahn and Jacobi families are orthogonal. Hence if $F$ does not satisfy (\ref{cis}) and either  $\ax+n_F<0$ or $\bx+n_F<0$, these measures can not be defined.
\item Some of the cases when (\ref{cis}) fails show other kind of degenerateness which again imply that the measures with respect to which our exceptional Hahn and Jacobi families are orthogonal are not defined (for instance because $\Omega_{\M,F}^{\ax,\bx,N}(n)=0$, for some $n=0,\cdots , N-n_F+1$, or because $\Omega_{\M,F}^{\ax,\bx}(\pm 1)=0$).
\item And moreover, we guess that when $F$ does not satisfy (\ref{cis}) and none of the above degenerateness happens, then the exceptional families defined from $F$ are particular cases of exceptional families defined from a set $\tilde F$ satisfying (\ref{cis}). For instance, as noticed above (\ref{gpe}), this happens
   for the sets  $F=\{1,\cdots, m_1,2m_1+1\}$, $m_1$ being a positive integer, which corresponds to the exceptional Legendre polynomials introduced in \cite{xle}.
\end{enumerate}

\bigskip

\noindent
\textit{Mathematics Subject Classification: 42C05, 33C45, 33E30}

\noindent
\textit{Key words and phrases}: Orthogonal polynomials. Exceptional orthogonal polynomial.
Hahn polynomials. Jacobi polynomials. Krall discrete polynomials.

 \end{document}